\documentclass{article}

\usepackage[english]{babel}
\usepackage[utf8]{inputenc}
\usepackage[centering]{geometry}

\usepackage{amsfonts}
\usepackage{amsmath}
\usepackage{amsthm}
\usepackage{amssymb}
\usepackage{mathtools}
\usepackage[mathscr]{eucal}
\usepackage{bbm}
\usepackage{enumitem}
\usepackage{fancyhdr}
\usepackage{longtable}
\usepackage{comment}

\usepackage{graphicx}
\usepackage[dvipsnames]{xcolor}
\usepackage{tikz}
\usepackage{tikzsymbols}

\usepackage{tikzrput}
\usetikzlibrary{decorations,decorations.markings}
\usetikzlibrary{decorations.pathreplacing,shapes.misc}
\usepackage{pgfplots}
\usetikzlibrary{math} 
\usetikzlibrary{intersections}  

\usepackage[hypertexnames=false, breaklinks,
   colorlinks=true, 
   citecolor=red!70!black,
   linkcolor=blue,
   anchorcolor=cyan,
   urlcolor=green!50!black,
   pagebackref=false,
   ]{hyperref}

\usepackage[breakable]{tcolorbox}
\tcbuselibrary{theorems}
\newtcolorbox{bluebox}[1][]%
{left=0mm, right=0mm, bottom=0mm, top=0mm, sharp corners, boxrule=.8pt, before skip=\topsep, after skip=\topsep, colback=cyan!5, colframe=cyan, coltitle=black, fonttitle=\bfseries, title=#1, breakable}
\tcbset{tcbox raise base, sharp corners, left=0mm,right=0mm,top=0mm,bottom=0mm,nobeforeafter, boxrule =.5pt} 
\tcbset{tcbox raise base, sharp corners, left=0mm,right=0mm,top=0mm,bottom=0mm,nobeforeafter, boxrule =.5pt}

\usepackage[color=yellow!30, linecolor=orange, colorinlistoftodos]{todonotes}

\newcommand{\R}{\mathbb{R}}
\newcommand{\C}{\mathbb{C}}

\newcommand{\N}{\mathbb{N}}
\newcommand{\mD}{\mathcal{D}}

\newcommand{\mK}{\mathcal{K}}
\newcommand{\mL}{\mathcal{L}}
\newcommand{\mM}{\mathcal{M}}
\newcommand{\mW}{\mathcal{W}}
\newcommand{\mU}{\mathcal{U}}
\newcommand{\rd}{\mathrm d}
\newcommand{\id}{I}
 
\newcommand{\I}{\mathrm{i}} 

\newcommand{\Hsym}{\mathcal H_{\mathrm{sym}} }
\newcommand{\Ssym}[1][]{S_{\mathrm{sym} #1}}  
\newcommand{\wGamma}{\widetilde \Gamma}  
\newcommand{\wB}{\widetilde B}  
\newcommand{\wC}{\widetilde C}  

\newcommand{\scalar}[2]{\langle #1,\, #2 \rangle}  
\newcommand{\bscalar}[2]{\big\langle #1,\, #2 \big\rangle}  
\newcommand{\Bscalar}[2]{\Big\langle #1,\, #2 \Big\rangle}  
\newcommand{\lrscalar}[2]{\left\langle #1,\, #2 \right\rangle}  

\DeclareMathOperator{\im}{Im}         
\DeclareMathOperator{\rk}{\mathrm rk}
\DeclareMathOperator{\range}{rg}
\DeclareMathOperator{\gen}{span}

\newcommand{\define}[1]{\emph{#1}}

\newcommand\aug{\fboxsep=-\fboxrule\!\!\!\fbox{\strut}\!\!\!}

\theoremstyle{plain}
\newtheorem{theorem}{Theorem}[section]
\newtheorem*{theorem*}{Theorem}
\newtheorem{lemma}[theorem]{Lemma}
\newtheorem*{lemma*}{Lemma}
\newtheorem{proposition}[theorem]{Proposition}
\newtheorem*{proposition*}{Proposition}
\newtheorem{corollary}[theorem]{Corollary}
\newtheorem*{corollary*}{Corollary}

\theoremstyle{definition}
\newtheorem{definition}[theorem]{Definition}
\newtheorem{remark}[theorem]{Remark}
\newtheorem*{remark*}{Remark}

\allowdisplaybreaks

\begin{document}

\title{Complete Non-Selfadjointness of Extensions of Symmetric Operators with Bounded Dissipative Perturbations}
\author{%
{Christoph Fischbacher}
\thanks{Department of Mathematics, Baylor University, Texas, USA, c\_fischbacher@baylor.edu},
{Andr\'es Felipe Pati\~no L\'opez}
\thanks{Departamento de Matem\'aticas, Universidad de los Andes, Bogot\'a, Colombia, af.patinol@uniandes.edu.co},
{Monika Winklmeier}%
\thanks{Departamento de Matem\'aticas, Universidad de los Andes, Bogot\'a, Colombia, mwinklme@uniandes.edu.co}
}
\maketitle

\begin{abstract}
   \noindent
   \begin{center}
   
    Using boundary triples, we develop an abstract framework to investigate the complete non-selfadjointness of the maximally dissipative extensions of dissipative operators of the form $S+iV$, where $S$  is symmetric with equal finite defect indices and $V$ is a bounded non-negative operator. Our key example is the dissipative Schr\"odinger operator $-\tfrac{d^2}{dx^2}+\I V$ on the interval.

   \end{center}

\end{abstract}
\bigskip

\smallskip


\bigskip\noindent
{\bf Acknowledgements.}
We would like to thank Pavel Kurasov who provided us with the proof of Theorem \ref{thm:sa+pot=cnsa}. Moreover, we would like to thank the anonymous referee for their very valuable and detailed feedback that helped us significantly improve the presentation of our results.
CF was supported in part by the National Science Foundation under grant DMS--2510063. He also appreciates the hospitality of the Universidad de Los Andes, where this work was initiated.
AFPL was supported the Facultad de Ciencias of the University of the Andes (INV-2023-178-2986 and INV-2024-181-3013). He is grateful for the hospitality of Baylor University.
MW was supported by the Facultad de Ciencias of the University of the Andes (INV-2025-213-3442).


\section{Introduction} 

In this paper, we provide an abstract framework to analyze the complete non-selfadjointness of maximally dissipative extensions of dissipative operators of the form $S+\I V$, where $S$ is symmetric with equal defect indices and $V$ is a bounded non-negative operator.

Dissipative operators play an important role in the description of open quantum systems \cite{Phil59, MT22} with numerous applications in fields such as (magneto-)hydrodynamics, lasers, or nuclear scattering. Dissipative operators with no non-trivial dissipative extensions are called \emph{maximally dissipative} and are of fundamental mathematical importance as they generate strongly continuous semigroups of contractions \cite{LP61}. 

A maximally dissipative operator is completely non-selfadjoint (cnsa) if there exists no non-trivial reducing subspace on which the operator is selfadjoint (see Section \ref{sec:cnsa} for precise definitions). If a maximally dissipative operator $A$ is not completely non-selfadjoint, then there exists a unique non-trivial reducing subspace $H_{sa}(A)$ such that $A$ is selfadjoint in $H_{sa}(A)$ and completely non-selfadjoint in $H_{sa}(A)^\perp=:H_{cnsa}(A)$. Accordingly, one can decompose $A=A_{sa}\oplus A_{cnsa}$ into its selfadjoint part $A_{sa}$ and its cnsa part $A_{cnsa}$. This means for the construction of a selfadjoint dilation of a maximally dissipative operator $A$ \cite{CKS21, BMNW20, BMNW24, SzNFBK10} one only has to consider the cnsa part $A_{cnsa}$. Determining $H_{sa}(A)$ and $H_{cnsa}(A)$ can be a challenging problem and for each particular operator whose complete non-selfadjointness has been previously studied \cite{Pav75, Pav76, Pav77, Al91, KNR03, RT11}, new specific approaches and methods had to be employed.

The description of maximally dissipative extensions of a given dissipative operator is much richer than that of selfadjoint extensions of a given symmetric operator: while within von Neumann theory, finding all selfadjoint extensions of a symmetric operator corresponds -- via the Cayley transform -- to finding all unitary extensions of an isometry, the Cayley transform $C_A: \range(A+i)\rightarrow\range(A-i)$ of a given dissipative operator $A$ is a contraction ($\|C_A\|\leq 1$). Since there is a one-to-one correspondence between all maximally dissipative extensions of $A$ and all contractive extensions of $C_A$ to the entire Hilbert space $H$ with $\ker(I-C_A)=\{0\}$ \cite{Phil59}, the increased complexity of the extension problem for $A$ (compared to the symmetric to selfadjoint case) is due to the larger number of choices available when extending a contraction to the entire Hilbert space \cite{Crand69, AG82} compared to extending an isometry to a unitary operator.

To illustrate this on a concrete example, let $H=L_2(0,\infty)$ and consider the standard minimal Laplacian in $H$:
\begin{equation*}
S:\quad\mD(S)=H^2_0(0,\infty)=\{f\in H^2(0,\infty)\: |\: f(0)=f'(0)=0\},\quad f\mapsto -f''\:,
\end{equation*}
whose adjoint $S^*$, the maximal realization of the Laplacian, is given by 
\begin{equation*}
S^*:\quad\mD(S^*)=H^2(0,\infty),\quad f\mapsto -f''\:.
\end{equation*}
The minimal operator $S$ is symmetric, has defect indices both equal to one and \emph{all} of its selfadjoint extensions are thus described by a real one-parameter family $\{S_r\}_{r\in\mathbb{R}\cup\{\infty\}}$, given by
\begin{equation} \label{eq:salapl}
S_r:\quad\mD(S_r)=\{f\in \mD(S^*)\: |\: f'(0)=rf(0)\}, \quad S_r=S^*|_{\mD(S_r)}\:,
\end{equation}
where we adapt the convention that $r=\infty$ corresponds to the Dirichlet realization of the Laplacian. 
Since $S$ is symmetric, i.e.\ $S\subseteq S^*$, it immediately follows that all of its selfadjoint extensions are also restrictions of $S^*$, i.e.\ $S\subseteq S_r=S^*_r\subseteq S^*$, and thus the action of any selfadjoint extension $S_r$ preserves the differential expression ``$-\tfrac{d^2}{dx^2}$".

On the side of Cayley transforms, the fact that there is only a real one-parameter family of selfadjoint extensions of $S$ corresponds to the fact that all unitary extensions of the Cayley transform $C_S$ of $S$ to the entire Hilbert space $H$ can be described in terms of unitary operators $U$ between the defect spaces, $U:\ker(S^*-i)\rightarrow\ker(S^*+i)$. Since in this concrete example, both defect spaces are just one-dimensional, there is only a single parameter $\psi\in[0,2\pi)$ describing all of these unitary maps $U_\psi:\ker(S^*-i)\rightarrow\ker(S^*+i)$ via $U_\psi n_+=e^{i\psi}n_-$, where $n_\pm$ are normalized vectors spanning $\ker(S^*\mp i)$, respectively. 

The problem of finding all maximally dissipative extensions of the symmetric operator $S$ still remains relatively easy. This has to do with the fact that all contractive extensions of the isometric Cayley transform $C_S$ of $S$ are still fully characterized by contractions that only map between $\ker(S^*-i)$ and $\ker(S^*+i)$ mentioned without proof in \cite{Phil59}, see also \cite[Lemma 3.1.1]{FDiss17} for a proof. Moreover, any maximally dissipative extension $\hat{S}$ of a symmetric operator $S$ has to be a restriction of its adjoint, i.e.\ $S\subseteq \hat{S}\subseteq S^*$ \cite{FDiss17}. In our particular example, all maximally dissipative extensions of the minimal Laplacian $S$ are still of the form $S_r$ as described in \eqref{eq:salapl}, however the parameter $r$ may now also have positive imaginary part, since
\begin{equation*}
\im\langle f,S_rf\rangle=\im(r)|f(0)|^2,\quad f\in\mD(S_r)\:.
\end{equation*}

Now, if one tries to find all maximally dissipative extensions of operators of the form $A=S+\I V$, where $S$ is symmetric and $V$ is a bounded non-negative operator\footnote{There are more general results for unbounded $V$ as long as $\mD(S)\cap\mD(V)$ is dense, cf. \cite{FNWproper2016, nonproper2018}.}, there is a much larger number of possibilities to choose from. If we view $A=S+\I V$ as the minimal operator, the appropriate choice of maximal operator $A_{max}$ is given by $A_{max}=S^*+\I V$ with $\mD(A_{max})=\mD(S^*)$. While the domain of \emph{any} maximally dissipative extension $\widehat{A}$ of $A$ still has to satisfy $\mD(\widehat{A})\subseteq\mD(A_{max})$ \cite{CP68, nonproper2018}, it is not necessarily true anymore that $\widehat{A}$ is a restriction of $A_{max}$. 

To modify the previous example of the minimal Laplacian $S$ in $L_2(0,\infty)$, let us consider a more general dissipative Schr\"odinger operator defined on the half-line. This means we add a bounded dissipative perturbation of the form $\I V$, where $V$ is non-negative, e.g.\ a multiplication by an a.e.\ non-negative function $V\in L_\infty(0,\infty)$:
\begin{equation*}
A:\quad\mD(A)=\mD(S)=H_0^2(0,\infty), \quad f\mapsto -f''+\I Vf\:.
\end{equation*}
 Clearly, the operator $A$ is dissipative. Of course, by choosing extensions of the form $A_r=S_r+\I V$, where $S_r$ is given by \eqref{eq:salapl} with $\im(r)\geq 0$, we have found a class of maximally dissipative extensions that are also restrictions of the maximal operator: $A\subseteq A_r\subseteq A_{max}$ thus preserving the differential expression ``$-\tfrac{d^2}{dx^2}+\I V$". Extensions with this property are called \emph{proper}. Indeed, since $S_r$ for $\im(r)\geq 0$ is maximally dissipative to begin with, adding a bounded dissipative perturbation $\I V$ preserves this property. 

Going beyond proper extensions, in addition to the complex-valued boundary parameter $r$, it is possible to introduce another Hilbert space-valued parameter $k\in\range(V^{1/2})$, which -- together with $r$ -- can be used to describe \emph{all} maximally dissipative extensions of $A$:
\begin{proposition}[\protect{\cite[Example 6.6]{nonproper2018}}]
All maximally dissipative extensions of $A$ are of the form $A_{r,k}$, where 
\begin{equation*}
A_{r,k}:\quad \mD(A_{r,k})=\mD(S_r),\quad (A_{r,k}f)(x)=-f''(x)+\I V(x)f(x)+f(0)k(x)
\end{equation*}
and $r\in\mathbb{C}$, $k\in\range(V^{1/2})$ satisfy 
\begin{equation} \label{eq:disscond}
    \im(r)\geq\frac{1}{4}\|V^{-1/2}k\|^2\:.
\end{equation}
Here, $V^{-1/2}$ denotes the inverse of $V^{1/2}$ on $\range(V^{1/2})$, which is a (possibly unbounded) selfadjoint operator in $\overline{\range(V^{1/2})}$ -- the closure of $\range(V^{1/2})$. 
\end{proposition}
The additional term $``+f(0)k"$ is referred to as non-local point-interaction \cite{AN07} and leads to interesting new problems and phenomena -- even in the first-order case \cite{FPP-RZ25}. A second look at Condition \eqref{eq:disscond} leads to the following observations:
\begin{itemize}
\item[(i)] In the case of selfadjoint boundary conditions, $\im(r)=0$, it necessarily follows that $k=0$ for an extension $A_{r,k}$ to be maximally dissipative.
\item[(ii)] More generally, the case of proper extensions corresponds to the choice $k=0$. To ensure that $A_{r}:=A_{r,k=0}$ is dissipative in this case, all that is required is the dissipativity of the boundary condition, i.e. $\im(r)\geq 0$, no matter what choice for $V\geq 0$ is made.
\item[(iii)] If we consider the non-proper case $k\neq 0$, Condition \eqref{eq:disscond} reflects an interplay between the dissipative boundary condition, the dissipativity of the potential, and the non-local point interaction. For example, if $r$ with $\im(r)>0$ and $k\in\range(V^{1/2}), k\neq 0$ are fixed, one can always rescale the potential to be more dissipative, $V\mapsto \mu V$ with $\mu>0$ large enough such that \eqref{eq:disscond} is satisfied. Similarly, one can increase the dissipativity of the boundary condition by increasing the imaginary part of $r$ or decrease the contribution of the non-local point-interaction by decreasing the norm of $k$ to ensure \eqref{eq:disscond} is satisfied. 

\end{itemize}

The question whether the extensions $A_{r,k}$ of the dissipative Schr\"odinger operator on the half-line are cnsa was the main focus of the article \cite{FNWcnsa2024}, where necessary and sufficient conditions to answer this question were found. In particular, if the inequality in \eqref{eq:disscond} is strict (the ``\emph{non-critical}" case), then $A_{r,k}$ has to be cnsa. Likewise, in the case of selfadjoint boundary conditions, $\im(r)=0, k=0$, if $V\neq 0$, the operator $A_{r,k}$ is cnsa. Hence, for there to be a non-trivial reducing selfadjoint subspace of $A_{r,k}$ it is necessary that $\im(r)=\|V^{-1/2}k\|^2/4$, which is referred to as the \emph{critical} case. In this case, if additional conditions are met \cite[Thm.\ 3.14]{FNWcnsa2024}, it is possible that $A_{r,k}$ has a one-dimensional reducing selfadjoint subspace. Moreover, given a fixed potential $V$, for any real $\lambda$ not in the essential spectrum of the minimal operator $A=S+\I V$, there exist unique $\tilde{r}\in\mathbb{C}^+$ and $\tilde{k}\in\range(V)\subseteq\range(V^{1/2})$ such that  $A_{\tilde{r},\tilde{k}}$ has a reducing selfadjoint subspace spanned by $V^{-1}\tilde{k}$. 

The methods used in \cite{FNWcnsa2024} did not rely too heavily in the explicit structure of the minimal operator $A$. In fact, the main part of the analysis only made use of the fact that (i) the defect index $\dim\ker(S^*-i)$ of $S$ is equal to one, (ii) the symmetric part of any maximally dissipative but not selfadjoint extension $S_r$ is equal to the minimal operator $S$, which has no non-trivial reducing selfadjoint subspaces, and (iii) the kernel of the dissipative perturbation $V$ is invariant under $S$, i.e.\ $Sh\in\ker(V)$ for any $h\in\mD(S)\cap\ker(V)$. Thus, a natural progression when studying the complete non-selfadjointness of dissipative differential operators is to consider the case of operators with higher defect indices such as dissipative Schr\"odinger operators on the interval or more generally, on metric graphs.  In \cite{KMN2025}, this was done for metric graphs with dissipative conditions at the vertices and bounded dissipative potentials. However, only proper extensions were considered, i.e.\ only extensions without non-local point-interactions.

While our results will be presented in the abstract setup with the help of boundary triples (cf.\ Sections \ref{sec:dissipativeoperators}, \ref{sec:cnsa}), the key example analyzed in this paper will be the minimal Laplacian $S$ on the interval with a dissipative bounded potential $\I V$. Due to the larger defect indices of $S$, the number of complex-valued and Hilbert space-valued parameters describing all maximally dissipative extensions of $S+iV$ increases significantly.
Moreover, it is not true anymore that the symmetric part of any maximally dissipative extension of $S$ is equal to $S$ itself and in particular, it may have non-trivial reducing subspaces in which it is selfadjoint, which further complicates the analysis.
\bigskip

This paper is organized as follows. 
In Section~\ref{sec:dissipativeoperators} we describe all maximally dissipative extensions of a symmetric operator $S$ with equal and finite defect indices $\eta$ with the help of boundary triples.
It turns out that the dissipativity of an $\eta$-dimensional restriction $T$ of $S^*$ is equivalent to the dissipativity of $BC^*$, where $B,C$ are the matrices used to define $T$.
If $B$ is invertible, this condition is equivalent to $\im( C^* B^{*-1}) \ge 0$.

These results are used to find criteria for the dissipativity of $\eta$-dimensional extensions $\widehat A$ of $A = S + \I V$ where $V\ge 0$ is a bounded linear operator. 
In contrast to the case when $V=0$, $\widehat A$ does not need to be a restriction of $S^*+\I V$.
However, it is always true that $\mD(\widehat A)$ coincides with the domain of a maximally dissipative extension $\widehat S$ of $S$, hence it is a subset of $\mD(S^*)$.
Under the assumption that the matrices $B, C$ describing $\widehat A$ are invertible, we give a condition for the dissipativity of $\widehat A$.
Instead of non-negativity of $\im (C^*B^{*-1})$ we now require that it be larger than or equal to a matrix $M_{\mathtt K}$ involving the terms describing the ``deviation'' of $\widehat A$ from $S^*+\I V$.
In Section~\ref{sec:maxdiss} we apply our results to the minimal Laplacian $S$ on the interval $(0,1)$. 
We classify all maximally dissipative extensions of $S$ and of $A=S+\I V$ for a bounded non-negative operator $V$.

In Section~\ref{sec:cnsa} we are concerned with the existence of selfadjoint subspaces of maximally dissipative operators $\widehat A$. 
To this end, we identify the largest symmetric extension $\Ssym$ of $S$ whose domain is still contained in $\mD(\widehat A)$ and view $\widehat A$ as a maximally dissipative extension of $\Ssym + \I V$, described by a new, lower dimensional, boundary triple $(\wGamma_0, \wGamma_1, \C^{\ell})$.
It turns out that now the corresponding boundary matrices $\wB$ and $\wC$ are invertible.
Under certain assumptions on $\ker V$ and the complete non-selfadjointness of $\Ssym$, we find upper bounds on the dimension of the selfadjoint subspace of $\widehat A$ in terms of $\rk( \im(C^* B^{*-1}) - \frac{1}{4}M_{\mathtt K})$.
In Section~\ref{sec:cnsa:Laplace} we apply our results on complete non-selfadjointness to the Laplacian on the interval.
If $V$ is a multiplication operator by a function which satisfies the 
mild assumption that it does not vanish on some non-trivial open interval, then the complete non-seladjointness of $\Ssym$ is not necessary to conclude that $\dim(H_{sa}(\widehat A)) \le \rk( \im(C^* B^{*-1}) - \frac{1}{4}M_{\mathtt K})$.

Our results can be applied to more general settings.
As a toy example we briefly discuss a star graph with a Kirchhoff condition at the central vertex.
\bigskip
\bigskip

{\bf Notation and conventions.}
All Hilbert spaces $H$ in this work are assumed to be complex.
We use the convention that the inner product on $H$ is linear in its second argument, and antilinear in its first argument, that is
$\scalar{c x}{y} = \overline c \scalar{x}{y}=\scalar {x}{\overline{c} y}$. 
For a linear operator $A$ in $\mathcal{H}$, let $\mD(A)$, $\range(A)$, and $\ker(A)$ denote its domain, range, and kernel, respectively.

We denote by $\C_\pm$ the upper and lower open half-planes.
For a bounded operator $M$ we define its imaginary part $\im M := \frac{1}{2\I}(M-M^*)$.

If $Y$ is a subspace of $X$ we write $W\subseteq X// Y$ to indicate that $W$ is a subspace of $X$ with $W\cap Y = \emptyset$.
An operator with tilde, e.g. $\widetilde S$, usually denotes an extension of another operator, while
an operator with hat, e.g. $\widehat S$, usually denotes a maximally dissipative extension of another operator.
\bigskip
\bigskip


\section{Maximally dissipative extensions} 
\label{sec:dissipativeoperators}

\subsection{Extensions of a symmetric operator}
Let $H$ be a Hilbert space and let $S$ be a closed symmetric operator in $H$ with finite and equal defect indices
\begin{equation}
   \eta :=
    \eta_+(S) := \dim \ker(S^*-\I) 
    = \eta_-(S) := \dim \ker(S^*+\I) < \infty.
\end{equation}
Let $(\Gamma_0, \Gamma_1, \mK)$ be a boundary triple for $S^*$.
This means that $\dim(\mK) = \eta$, 
$(\Gamma_0,\Gamma_1): \mD(S^*)\to \mK\oplus\mK$ is surjective and that the so-called abstract Green's identity 
\begin{equation}
   \scalar{f}{S^* g} - \scalar{S^* f}{g}
   =
   \scalar{\Gamma_0 f}{\Gamma_1 g} - \scalar{\Gamma_1 f}{\Gamma_0 g},
   \qquad f,g\in\mD(S^*)
\end{equation}
holds.
Moreover, $\mD(S) = \ker(\Gamma_0, \Gamma_1)$.
We refer the reader to \cite{BehrndtHassideSnoo} for more details.

Note that $\mD(S^*) = \mD(S) \dot + \mU$ where $\mU$ is a $2\eta$-dimensional subspace of $\mD(S^*)$ which is isomorphic to $\C^{2\eta}$ via the map 
\begin{equation*}
   u \mapsto
   \Gamma u := 
   \begin{pmatrix}
      \Gamma_0 u \\ \Gamma_1 u
   \end{pmatrix}.
\end{equation*}
Therefore, for any $0\le \eta' \le 2\eta$ there exists a bijection between the $\eta'$-dimensional subspaces $\mW$ of $\C^{2\eta}$ and the set of all operators $T$ such that $S\subseteq T \subseteq S^*$ and $\dim( \mD(T)/\mD(S) ) = \eta'$. 
For our work, the case $\eta'=\eta$ will be the most important.

Clearly, every $\eta$-dimensional subspace of $\C^{2\eta}$ is the kernel of an operator of
a matrix of the form 
$(B|-C)$ where $B$ and $C$ are $\eta\times \eta$ matrices such that $\rk(B|-C) = \eta$.
Hence, $S\subseteq T \subseteq S^*$ is an $\eta$-dimensional restriction of $S^*$ if and only if it is of the form $T= T_{B,C}$ where 
\begin{equation}
   \label{eq:extensionsBT}
   \mD(T_{B,C}) = 
   \left\{ f \in \mD(S^*) : B\Gamma_0 f = C \Gamma_1 f \right\},
   \quad
   T_{B,C} f = S^*f
\end{equation}
with $\eta\times \eta$ matrices $B,C$ such that $\rk(B|-C) = \eta$.
Next we describe the domain of the adjoint of $T_{B,C}$;
cf. \cite[Sec.~1.10]{BehrndtHassideSnoo}.
\begin{lemma}
   \label{lem:AdjointBT}
   Let $T = T_{B,C}$ be as above.
   The adjoint operator $T^*$ is the restriction of $S^*$ to the domain
   \begin{align}
       \label{eq:AdjointBC}
       \mD(T^*) 
       &= 
       \left\{ g\in \mD(S^*) : \exists\, \vec a \in \C^\eta\ \text{ s.t.\ } 
       \Gamma_0 g = C^* \vec a
       ,\
       \Gamma_1 g = B^* \vec a    
       \right\}.
   \end{align}
   If $g\in\mD(T^*)$, then 
   $\vec a = (CC^*+BB^*)^{-1}(C\Gamma_0 + B\Gamma_1)g $.
   \smallskip

   If $B,C$ are invertible, then 
   $T^* = T_{C^{*-1}, B^{*-1}}$ with
   \begin{align}
       \label{eq:AdjointBCinvertible}
       \mD(T^*) 
       &= 
       \left\{ g\in \mD(S^*) : 
       B^{*-1} \Gamma_1 g = C^{*-1}\Gamma_0 g 
       \right\}.
   \end{align}
\end{lemma}
\begin{proof}
   By assumption, we have $S\subseteq T \subseteq S^*$, hence also $S\subseteq T^* \subseteq S^*$.
   Therefore, $g\in \mD(T^*)$ if and only if $g$ belongs to $\mD(S^*)$ and if for every $f\in\mD(T)$
   \begin{align*}
      0 &=  
      \scalar{S^*g}{f} - \scalar{g}{Tf}
      =  \scalar{S^* g}{f} - \scalar{g}{S^* f}
      =  \scalar{\Gamma_1 g}{\Gamma_0 f} - \scalar{\Gamma_0 g}{\Gamma_1 f}
      \\
      &=  \Big( (\Gamma_1 g)^* | -(\Gamma_0 g)^* \Big)
      \begin{pmatrix}
	 \Gamma_0 f \\ \Gamma_1 f
      \end{pmatrix}
   \end{align*}
   where the last line has to be read as multiplication of matrices and $(\Gamma_j g)^*$ is the transposed and complex conjugated vector of $(\Gamma_j g)$.
   Since this holds for every $f\in\mD(T)$, it follows that $g\in\mD(T^*)$ if and only if the (row) vector 
   $\big( (\Gamma_1 g)^*,\ -(\Gamma_0 g)^* \big)$ is a linear combination of the rows of the $\eta\times 2\eta$ matrix $(B| -C)$, that is, if and only if 
   $\big( (\Gamma_1 g)^* | -(\Gamma_0 g)^* \big)
   = \vec a^* (B|-C)$ for some $\vec a\in\C^\eta$.
   Writing this as column vectors, we obtain
   \begin{equation*}
      \begin{pmatrix}
	 \Gamma_1 g \\ \Gamma_0 g 
      \end{pmatrix}
      = 
      \begin{pmatrix}
	 B^* \\ C^*
      \end{pmatrix}
      \vec a
      ,
      \qquad\text{hence}\qquad
      \Gamma_1 g = B^* \vec a
      \ \text{ and }\
      \Gamma_0 g = C^* \vec a
   \end{equation*}
   which shows \eqref{eq:AdjointBC}.
   The rank condition on $(B|-C)$ ensures the invertibility of $CC^*+BB^*$, hence the formula for $\vec a$ follows.
   If $B$ and $C$ are invertible, then \eqref{eq:AdjointBCinvertible} follows easily from \eqref{eq:AdjointBC}.
\end{proof}

For later use, we also note the following.
\begin{lemma}
\label{rem:GreenStuffTBC}
   \begin{enumerate}[label={\upshape(\roman*)}]

   \item \label{item:GreenStuffTBC:i}
   
   Let $\mW\subseteq \mD(S^*)//\mD(S)$ with basis $\{w_1,\, \dots,\, w_k\}$. 
   Then for every $w = a_1 w_1 + \dots + a_k w_k\in \mW$,
   \begin{align}
       \label{eq:wS*w}
       \im \scalar{w}{S^*w}
       &= \im \lrscalar{
       \begin{pmatrix}
           a_1 \\ \vdots \\ a_k
       \end{pmatrix}
       }{
       \Big( \scalar{\Gamma_0 w_r}{\Gamma_1 w_s} \Big)_{r,s=1}^k
       \begin{pmatrix}
           a_1 \\ \vdots \\ a_k
       \end{pmatrix}
       }
   \end{align}
   where the $k\times k$ matrix can be written as the matrix product
   \begin{align}
       \label{eq:matrixproduct}
       \Big( \scalar{\Gamma_0 w_r}{\Gamma_1 w_s} \Big)_{r,s=1}^k
       =
       \begin{pmatrix}
           (\Gamma_0 w_1)^* \\ \vdots \\ (\Gamma_0 w_k)^*
       \end{pmatrix}
       \Big(  \Gamma_1 w_1 | \dots | \Gamma_1 w_k \Big)
   \end{align}
   if $(\Gamma_0w_j)^*$ are understood as row vectors.

   \item \label{item:GreenStuffTBC:ii}
   Assume that $B$ is invertible. Then, for every $f,g\in\mD(T_{B,C})$ we have that
   \begin{align}
      \label{eq:GreenBCinvertible}
      \scalar{f}{T_{B,C}g} - \scalar{T_{B,C}f}{g}
      = \scalar{\Gamma_1 f}{ (C^*B^{*-1} - B^{-1}C) \Gamma_1 g}.
   \end{align}

    \item If $B$ is invertible, then
    $\Gamma_1|_{\mD(T_{B,C})}$ is surjective and if $C$ is invertible, then $\Gamma_0|_{\mD(T_{B,C})}$ is surjective.
   \end{enumerate}
\end{lemma}
\begin{proof}
   \begin{enumerate}[label={\upshape(\roman*)}]
   
   \item
 
   The claim follows from an easy calculation:
   \begin{align*}
       \im \scalar{w}{S^*w}
       & = \im \scalar{\Gamma_0 w}{\Gamma_1 w}
       = \im \lrscalar{ \sum_{j=1}^k a_j \Gamma_0 w_j}{ \sum_{j=1}^k a_j\Gamma_1 w_j}
       \\
       &= \im \sum_{r=1}^k \sum_{s=1}^k \overline{ a_r } a_s \lrscalar{  \Gamma_0 w_r}{\Gamma_1 w_s}.
   \end{align*}

   \item 
   Let $f,g\in\mD(T_{B,C})$. Then
   \begin{align*}
      \scalar{f}{T_{B,C}g} - \scalar{T_{B,C}f}{g}
      &= \scalar{f}{S^* g} - \scalar{S^* f}{g}
      = \scalar{\Gamma_0 f}{\Gamma_1 g} - \scalar{\Gamma_1 f}{\Gamma_0 g}
      \\
      &= \scalar{B^{-1}C\Gamma_1 f}{\Gamma_1 g} - \scalar{\Gamma_1 f}{B^{-1}C\Gamma_1 g}
      \\
      &= \scalar{\Gamma_1 f}{C^*B^{*-1}\Gamma_1 g} - \scalar{\Gamma_1 f}{B^{-1}C\Gamma_1 g}
      \\
      &= \scalar{\Gamma_1 f}{ (C^*B^{*-1} - B^{-1}C) \Gamma_1 g}.
   \end{align*}
   
   \item
    We prove the claim for $\Gamma_1$. The proof for $\Gamma_0$ is similar.
    Let $u\in\C^\eta$. 
    Then the surjectivity of $(\Gamma_0,\Gamma_1)$ in $\mD(S^*)$ guarantees the existence of $f_1,\, f_0\in \mD(S^*)$ such that $(\Gamma_0, \Gamma_1)f_1 = (0, u)$, and 
    $(\Gamma_0, \Gamma_1)f_0 = ( B^{-1} C \Gamma_1 f_1, 0)$. 
    Set $f := f_0 + f_1$. 
    Then $f\in \mD(T_{B,C})$ because $B\Gamma_0 f = B\Gamma_0 f_0 = C\Gamma_1 f_1 = C\Gamma_1 f$ and $\Gamma_1 f = u$.
    \qedhere
   \end{enumerate}
   
\end{proof}

\smallskip

Instead of writing the domain of $T_{B,C}$ as a \emph{restriction} of the domain of $\mD(S^*)$ as in \eqref{eq:extensionsBT}, we can equivalently view it as an \emph{extension} of the domain of $S$
\begin{equation*}
   \mD(T_{B,C}) = \mD(S) \dot + \gen \{ u_1, \dots, u_\eta\}
\end{equation*}
where $u_1, \dots, u_\eta \in \mD(S^*)$ are such that $\gen\{ \Gamma u_1, \dots, \Gamma u_\eta\} = \ker(B|-C)$.
The latter implies in particular that the $u_1, \dots, u_\eta$ are linearly independent modulo $\mD(S)$.

\medskip

It should be noted that there is a one-to-one relation between the $\eta$-dimensional subspaces of $\C^{2\eta}$ and the $\eta$-dimensional extensions $T$ of $S$ with $S\subseteq T \subseteq S^*$.
However, the matrices $B,C$ in \eqref{eq:extensionsBT} are not unique.
In fact, $S_{B,C} = S_{B', C'}$ if and only if $\ker(B|-C) = \ker(B'|-C')$, which is the case if and only if there exists an invertible $\eta\times \eta$ matrix $E$ such that $B' = EB$ and $C' = EC$.

\subsection{Dissipative extensions of a symmetric operator}
Recall that a densely defined linear operator $T$ on a Hilbert space is called \define{dissipative} if 
\begin{equation*}
   \im \scalar{x}{Tx} \ge 0,
   \qquad x\in\mD(T).
\end{equation*}
It is called \define{maximally dissipative} if it has no non-trivial dissipative extension.
Note that every symmetric operator is dissipative and every selfadjoint operator is maximally dissipative.
\smallskip

If $T$ is dissipative, then $T-z$ is injective for every $z\in\C_-$.
A dissipative operator $T$ is maximally dissipative if $T-z$ is surjective for one (hence for all) $z\in\C_-$.
If both $T$ and $-T^*$ are dissipative, then both are in fact maximally dissipative.
\smallskip

If the defect index $\eta_+(T) = \dim\ker(T^*-\I)$ is finite, then a dissipative extension $\widehat T$ of $T$ is maximally dissipative if and only if $\dim( \mD(\widehat T)/\mD(T) ) = \eta_+(T)$.
\smallskip

In this work we will be mostly concerned with operators of the form $A= S + \I V$, where $S$ is symmetric and $V$ a bounded non-negative operator. In this case, a dissipative extension $\widetilde A$ of $A$ is called a \define{proper dissipative extension} if 
$\widetilde A \subseteq S^*+\I V=:A_{max}$.
Otherwise it is called a \define{nonproper dissipative extension of} $A$.
We will see in Theorem~\ref{thm:ChristophMaxDiss} that not every dissipative extension of $A$ needs to be a restriction of $A_{max}$.
However, every dissipative extension of a symmetric operator ($V=0$) is a proper extension as the next proposition shows.

\begin{proposition}[\protect{\cite[Lemma 6.2]{nonproper2018}, \cite[Theorem 4.19]{FNWproper2016}}]
   \label{prop:DissExtSymmetric}
   Let $S$ be a symmetric and $V$ a bounded non-negative operator. 

   \begin{enumerate}[label={\upshape(\roman*)}]

      \item
      If $\widetilde S$ is a dissipative extension of $S$, then $\widetilde S \subseteq S^*$.

      \item 
      The operator $A:= S + \I V$ is dissipative, $A^* = S^* - \I V$ and every dissipative extension $\widehat A$  of $A$ satisfies $\mD(\widehat A)\subseteq \mD(S^*)$.

   \end{enumerate}
\end{proposition}

If $V$ is a symmetric non-negative operator on $H$, then it has a unique non-negative selfadjoint square root $V^{1/2}$ and $H$ is the direct orthogonal sum $H = \ker V^{1/2} \oplus \overline{\range V^{1/2}}$.
So we can define $V^{-1/2}: \range(V^{1/2}) \to \overline{\range(V^{1/2})}$, which is a selfadjoint operator in $\overline{\range(V^{1/2})}$.

The next theorem describes all dissipative extensions of such operators.

\begin{theorem} [\protect{\cite[Theorem 6.3]{nonproper2018}}]
   \label{thm:ChristophMaxDiss}
   Let $S$ be a symmetric and $V$ a non-negative bounded operator and set $A = S + \I V$.
   An extension $\widetilde A$ of $A$ is dissipative if and only if there exists a subspace $\mW \subseteq \mD(S^*)// \mD(S)$ and a linear operator 
   \begin{equation*}
      \mL: \mW \to \range(V^{1/2})
   \end{equation*}
   such that 
   \begin{equation}
      \label{eq:basiccrit}
      \im \scalar{w}{S^*w} \ge \frac{1}{4}  \|V^{-1/2} \mL w \|^2,
      \qquad w\in \mW
   \end{equation}
   and $\widetilde A$ is of the form
   \begin{equation*}
      \widetilde A = A_{\mW, \mL}
      \qquad \text{with}\qquad
      \mD(A_{\mW, \mL}) = \mD(A) \dot + \mW,
   \end{equation*}
   and
   \begin{equation*}
      \widetilde A(f_0+w) 
      = A_{\mW, \mL}(f_0+w)
      = (S^* + \I V)(f_0+w) + \mL w
   \end{equation*}
   for all $f_0\in \mD(A)$ and $w\in \mW$.
\end{theorem}

It should be noted that the map $\mD(\widehat A)\to H,\ f_0+w \mapsto w$ is in general unbounded.

\begin{corollary}
   \label{cor:domains}
   Let $\mW\subseteq \mD(S^*)//\mD(S)$ 
  and set $S_{\mW} := S^*|_{\mD(A_{\mW, \mL})}$.
   Then $\mD(S)\dot + \mW$ is the domain of a dissipative extension of $S+\I V$ if and only if it is the domain of a dissipative restriction of $S^*$.

   In particular, a dissipative extension $A_{\mW,\mL}$ is maximally dissipative if and only if $S_{\mW}$ is a maximally dissipative extension of $S$. 
   If $\eta_+(S) < \infty$ this is the case if and only if 
   $\dim(\mD(A_{\mW,\mL})/\mD(A)) = \dim( \mD(S_{\mW})/\mD(S) ) = \eta_+(S)$.
\end{corollary}
\begin{proof}
   If $A_{\mW,\mL}$ is dissipative, then so is $A_{\mW,0}$, and consequently also
   $S_{\mW}$, by \eqref{eq:basiccrit}.
   Then clearly $S_{\mW}$ is a dissipative restriction of $S^*$.
   On the other hand, if $\widetilde S$ is a dissipative extension of $S$ with domain $\mD(\widetilde S ) = \mD(S)\dot + \mW$, then it is a restriction of $S^*$ by Proposition~\ref{prop:DissExtSymmetric}, and 
   $A_{\mW,0}$ with domain $\mD(\widetilde S)$ is a dissipative extension of $A=S+\I V$ (and a dissipative restriction of $S^*+\I V$).
   Cf. \cite[Corollary 5.9]{nonproper2018}.
\end{proof}

The corollary shows that in order to identify all possible domains of dissipative extensions of $A=S+ \I V$, it is enough to find all possible domains of dissipative extensions of $S$.
Let us assume that $S$ has equal and finite defect indices. 
Since all maximally dissipative extensions of $S$ are $\eta$-dimensional restrictions of $S^*$, they are necessarily of the form $T_{B,C}$ as in \eqref{eq:extensionsBT}.
It depends on the matrices $B,C$ whether the extension $T_{B,C}$ is in fact dissipative.

\begin{lemma}
    \label{lem:MaxDissCritBT}
    Let $S$ be a symmetric operator on a Hilbert space $H$ with $\eta := \eta_+(S) = \eta_-(S) < \infty$.
    Let $(\Gamma_0, \Gamma_1, \mK)$ be a boundary triple for $S^*$ and, as in  \eqref{eq:extensionsBT}, let 
    \begin{equation}
       \label{eq:extensionsBTinvertible}
       \mD(T_{B,C}) = 
       \left\{ f \in \mD(S^*) : B\Gamma_0 f = C \Gamma_1 f \right\},
       \quad
       T_{B,C} f = S^*f.
    \end{equation}
    Then $T_{B,C}$ is maximally dissipative if and only if 
    $\rk(B|-C)=\eta$ and
    \begin{equation}
    \label{eq:DissCondBC}
        \im (B C^*) \ge 0.
    \end{equation}
    If $B$ is invertible, then this is equivalent to 
    \begin{equation}
        \im (C^*B^{*-1}) \ge 0.
    \end{equation}
\end{lemma}
\begin{proof}
   Recall that a dissipative extension $T$ of $S$ is maximally dissipative if and only if it is an $\eta$-dimensional extension of $S$.
   So let us assume that $\rk(B|-C)=\eta$.
   We have to show that the extension $T_{B,C}$ is dissipative if and only if \eqref{eq:DissCondBC} holds.
   To this end, we will show that \eqref{eq:DissCondBC} is equivalent to the dissipativity of $-T_{B,C}^*$.
   Then $-T_{B,C}^*$ is maximally dissipative, and therefore $T_{B,C} = -(-T_{B,C}^*)^*$ is so, too.
   So it is sufficient to check when
   $\frac{1}{2\I} (\scalar{f}{-T_{B,C}^*f} - \scalar{-T_{B,C}^*f}{f}) \ge 0$ for all $f\in\mD(T_{B,C}^*)$.
   Let
   $\vec a \in \C^\eta$ such that 
   $\Gamma_0 f = C^* \vec a$, $\Gamma_1 f = B^* \vec a$.
   Then, 
   \begin{align*}
      \frac{1}{2\I} \big( \scalar{f}{-T_{B,C}^*f} - \scalar{-T_{B,C}^*f}{f} \big)
      & = \frac{1}{2\I} \big( -\scalar{f}{S^*f} + \scalar{S^*f}{f} \big)
      = \frac{1}{2\I} \big( -\scalar{\Gamma_0 f}{\Gamma_1 f} + \scalar{\Gamma_1 f}{\Gamma_0 f} \big)
      \\
      &= \frac{1}{2\I} \scalar{\vec a}{ (BC^* - CB^*) \vec a}
      = \bscalar{\vec a}{ \textstyle \frac{1}{2\I} (BC^* - CB^*) \vec a}
      \\
      & = \bscalar{\vec a}{ \im(BC^*) \vec a}.
   \end{align*}
   This is true for every $\vec a\in\C^\eta$ if and only if $\im(BC^*)\ge 0$.
   If $B$ is invertible, we note that for all $\vec a\in\C^\eta$
   \begin{equation*}
      \scalar{\vec a}{(BC^*-CB^*) \vec a}
      = \scalar{B^* \vec a}{ ( C^*B^*{}^{-1} - B^{-1} C) B^*\vec a},
   \end{equation*}
   hence $\im (BC^*)\ge 0$ if and only if $\im (C^* B^*{}^{-1})\ge 0$.
\end{proof}
\begin{remark}
Note that Criterion \eqref{eq:DissCondBC} is invariant under replacing the operators $(B,C)$ by $(EB,EC)$, where $E$ is invertible. This can be seen from
\begin{align*}
\im((EB)(EC)^*)=\frac{1}{2\I}\left(EB(EC)^*-(EC)(EB)^*\right)=E\left(\frac{1}{2\I}\left(BC^*-CB^*\right)\right)E^*=E\im(BC^*)E^*\:,
\end{align*}
which is non-negative if and only if $\im(BC^*)\geq 0$.
\end{remark}

Let $S$ be symmetric with equal and finite defect indices $\eta_+(S) = \eta_-(S) = \eta$ and let $V\ge 0$ be a bounded operator on $H$.
Let $\widehat A$ be a maximally dissipative extension of $A = S+\I V$.
Then $\widehat S = S^*|_{\mD(\widehat A)}$ is a maximally dissipative extension of $S$, so we can write the domain of $\widehat A$ as
\begin{align}
    \label{eq:mUBC}
    \mD(\widehat A) 
    &= \left\{ f\in \mD(S^*) : B\Gamma_0 f = C \Gamma_1 f\right\}
    = \mD(S) \dot +\, \mU
\end{align}
where $\mU$ is an $\eta$-dimensional subspace $\mU\subseteq \mD(S^*)//\mD(S)$ and $(B|-C)$ is a $\eta\times 2\eta$ matrix of rank $\eta$ with $\im(BC^*)\ge 0$.
We also know that $\widehat A$ is of the form 
\begin{align}
    \widehat A f = (S^*+\I V) f + \mL P f
\end{align}
where $P$ is the projection of $\mD(\widehat A)$ onto $\mU$ along $\mD(S)$
and $\mL:\mU\to \range(V^{1/2})$.
Note also that $(\Gamma_0, \Gamma_1)|_\mU$ is injective because 
$(\Gamma_0 f, \Gamma_1 f) = (\Gamma_0 g, \Gamma_1 g)$ implies that $f-g\in\ker(\Gamma_0,\Gamma_1) = \mD(S)$.
\smallskip

Let $u_1,\, \dots,\, u_\eta$ be a basis of $\mU$ from \eqref{eq:mUBC} and set
\begin{equation}
\label{eq:UK}
    \mathtt U :=
    \begin{pmatrix}
        u_1 \\ \vdots \\ u_\eta
    \end{pmatrix},
    \qquad
    \mathtt K :=
    \begin{pmatrix}
        k_1 \\ \vdots \\ k_\eta
    \end{pmatrix}
    := \mL \mathtt  U
    =
    \begin{pmatrix}
        \mL u_1 \\ \vdots \\ \mL u_\eta
    \end{pmatrix}.
\end{equation}
and 
\begin{equation}
   \label{eq:defMK}
   M_{\mathtt K} :=
   \big( \scalar{V^{-1/2} k_r}{V^{-1/2} k_s} \Big)_{r,s=1}^\eta
   = \big( \scalar{V^{-1/2} \mL u_r}{V^{-1/2} \mL u_s} \Big)_{r,s=1}^\eta
\end{equation}
\begin{proposition}
   \label{prop:dissCond}
   The dissipativity condition \eqref{eq:basiccrit} is equivalent to 
   \begin{equation}
   \label{eq:dissCond}
      \im \Big( \Big( \scalar{\Gamma_0 u_r}{ \Gamma_1 u_s} \Big)_{r,s=1}^\eta \Big)
      \ge
      \frac{1}{4}
      M_{\mathtt K}.
   \end{equation}
\end{proposition}
\begin{proof}
    Let $w = a_1 u_1 + \dots + a_\eta u_\eta \in \mW$ and let 
    $\vec a =  \begin{pmatrix} a_1 \\ \vdots \\ a_\eta
    \end{pmatrix}
    $.
    Then, by Lemma~\ref{rem:GreenStuffTBC}~\ref{item:GreenStuffTBC:i}, the left hand side in \eqref{eq:basiccrit} is equal to 
    $\im\Bscalar{\vec a}{ \Big( \scalar{\Gamma_0 u_r}{ \Gamma_1 u_s} \Big)_{r,s=1}^\eta \vec a}$
    and a similar calculation shows that the right hand side in \eqref{eq:basiccrit} is 
    \begin{equation*}
       \frac{1}{4} \| V^{-1/2} \mL w \|^2
       = \frac{1}{4} \sum_{r, s = 1}^\eta \overline{a_r} a_s \Bscalar{V^{-1/2} \mL u_r}{V^{-1/2} \mL u_s}
       =  \frac{1}{4}\Bscalar{\vec a}{ M_{\mathtt K} \vec a}.
       \qedhere
    \end{equation*}
\end{proof}

Now we describe the adjoint of $\widehat A$ in terms of the matrices $B$ and $C$ under the assumption that $B$ is invertible.
Clearly, a similar result can be formulated for invertible $C$.
The proof is analogous to that of Lemma~\ref{lem:AdjointBT}.

\begin{theorem}
    \label{thm:AdjointOfA}
    Let $\widehat A$ be a maximally dissipative extension of $A = S + \I V$ as in Theorem~\ref{thm:ChristophMaxDiss}.
    Then the domain of $\widehat A$ is of the form
    \begin{align}
        \mD(\widehat A) 
        = \{ f\in\mD(S^*) : B\Gamma_0 f = C\Gamma_1 f\}
        = \mD(S) \dot + \mU
    \end{align}
    where $\rk(B|-C)=\eta$ and 
    $\mU\subseteq \mD(S^*)//\mD(S)$ is of dimension $\eta$.
    Let $P$ be the projection of $\mD(\widehat A)$ onto $\mU$ along $\mD(S)$ and 
    let $\mL:\mU \to\range(V^{1/2})$ so that 
    \begin{align}
       \widehat A f = (S^* + \I V)f + \mL P f.
    \end{align}
    If $B$ is invertible, then $\Gamma_1|_{\mU}$ is surjective and we can choose a basis $u_1,\, \dots,\, u_\eta$ of $\mU$ such that $\Gamma_1 u_j = \vec e_j\in\C^\eta$.
    Let $k_j = \mL u_j$.
    Then the adjoint of $\widehat A$ is 
    \begin{align}
        \widehat A^*:
        \quad\mD(\widehat A^*) =
        \left\{ g\in \mD(S^*) :
          \Gamma_0 g - C^*B^{*-1}\Gamma_1 g + 
          \begin{pmatrix} \scalar{k_1}{g} \\ \vdots \\ \scalar{k_\eta}{g}
        \end{pmatrix} = 0
        \right\},
        \quad 
        \widehat A^* g = S^* g - \I Vg .
    \end{align}
\end{theorem}
It is possible to write down an expression for $\mD(\widehat A^*)$ even if $B$ is not invertible. 
However, the expression is more tedious and it is not needed for the purposes of our present work.

\begin{proof}
    We know from \cite{FDiss17, CP68} that $\mD(\widehat A)\subseteq \mD(S^*)$.
    Let $g\in \mD(S^*)$.
    Then for every  $f\in\mD(\widehat A)$
   \begin{align*}
      \scalar{g}{\widehat Af}
      &= \scalar{g}{S^* f}
      + \scalar{g}{(\I V + \mL P) f}
      \\
      &= \scalar{S^* g}{f} 
      + \scalar{g}{\I V f}
      + \scalar{\Gamma_0 g}{\Gamma_1 f} - \scalar{\Gamma_1 g}{\Gamma_0 f}
      + \scalar{g}{\mL P f}.
   \end{align*}
   The first two terms are continuous in $f$, so let us consider only the last three terms.
   
   Since $B$ is invertible, we know that $\Gamma_1|_{\mU}$ is surjective by Lemma~\ref{rem:GreenStuffTBC}, and we can choose 
   a basis $u_1,\, \dots,\, u_\eta$ of $\mU$ such that $\Gamma_1 u_j = \vec e_j\in\C^\eta$.
   Let $\mathtt U$ and $\mathtt K$ as in \eqref{eq:UK}.
   Since $\Gamma_1|_{\mD(S)} = 0$, we have that 
   $Pf = (\Gamma_1 f)^t\ \mathtt U 
   = (\Gamma_1 f)_1 u_1 + \cdots + (\Gamma_1 f)_\eta u_\eta
   $
   and $\mL Pf = \mL (\Gamma_1 f)^t\, \mathtt U 
   = (\Gamma_1 f)^t ( \mL \mathtt U )$ 
   and we obtain
 
   \begin{align*}
      \scalar{g}{\mL P f}
      &= \scalar{g}{(\Gamma_1 f)^t\ \mathtt K}
      = \scalar{g}{(\Gamma_1 f)_1 k_1 + \cdots + (\Gamma_1 f)_\eta k_\eta}
      \\
      &= (\Gamma_1 f)_1\scalar{g}{ k_1 } + \cdots + (\Gamma_1 f)_\eta \scalar{g}{k_\eta}
      = (\Gamma_1 f)_1 \overline{\scalar{ k_1 }{g} } + \cdots + (\Gamma_1 f)_\eta \overline{\scalar{k_\eta}{g} }
      \\
      &= \lrscalar{ 
      \begin{pmatrix} \scalar{k_1}{g} \\ \vdots \\ \scalar{k_\eta}{g}
      \end{pmatrix}}{ \Gamma_1 f }.
   \end{align*}
   Together with 
   \begin{align*}
      \scalar{\Gamma_1 g}{\Gamma_0 f}
      = \scalar{\Gamma_1 g}{B^{-1}C\Gamma_1 f}
      = \scalar{C^*B^{*-1}\Gamma_1 g}{\Gamma_1 f}
   \end{align*}
   we obtain
   
   \begin{align*}
      \scalar{\Gamma_0 g}{\Gamma_1 f} - \scalar{\Gamma_1 g}{\Gamma_0 f}
      + \scalar{g}{\mL P f}
      &= \lrscalar{\Gamma_0 g - C^*B^{*-1}\Gamma_1 g + 
      \begin{pmatrix} \scalar{k_1}{g} \\ \vdots \\ \scalar{k_\eta}{g}
      \end{pmatrix}}{ \Gamma_1 f }.
   \end{align*}
   This is continuous in $f$ if and only if the first argument of the inner product is $0$.
   Therefore $g\in\mD(\widehat A^*)$ if and only if 
   $\Gamma_0 g - C^*B^{*-1}\Gamma_1 g + 
   \begin{pmatrix} \scalar{k_1}{g} \\ \vdots \\ \scalar{k_\eta}{g}
   \end{pmatrix} = 0$.
\end{proof}

\begin{corollary}
    \label{cor:MaxDissCritABT}
    In the situation above, assume that the matrix $B$ in \eqref{eq:mUBC} is invertible.
    Then 

    \begin{equation}
        \label{eq:CBinCoordinates}
        \im(C^* B^{*-1}) = \im (\scalar{\Gamma_0 u_r}{\Gamma_1 u_s} )_{r,s=1}^\eta
    \end{equation}
    and the dissipativity condition $\im \scalar{w}{S^*w} \ge \frac{1}{4}\| V^{-1/2}\mL w\|^2$ for all $w\in \mU$ from \eqref{eq:basiccrit} 
    is equivalent to
    \begin{equation}
        \label{eq:DissCritBCV}
        \im (C^*B^{*-1}) \ge 
        \frac{1}{4} M_{\mathtt K}
    \end{equation}
    where $u_j\in\mU$ such that $\Gamma_1 u_j = \vec e_j \in \C^\eta$ for $j=1,\, \dots, \eta$.
\end{corollary}
\begin{proof}

    The equality in \eqref{eq:CBinCoordinates} follows from Lemma~\ref{rem:GreenStuffTBC}~\ref{item:GreenStuffTBC:i} and \ref{item:GreenStuffTBC:ii}
    and \eqref{eq:DissCritBCV} follows from \eqref{eq:CBinCoordinates} and Proposition~\ref{prop:dissCond}.
    
    %
\end{proof}

\section{Maximally dissipative extensions of Schr\"odinger operators on a bounded interval} 
\label{sec:maxdiss}

In this section, our Hilbert space is $H = L_2(0,1)$.
We use the standard notation $H^1(0,1)$, etc. for the Sobolev spaces.
\smallskip

Let $V$ be a bounded non-negative linear operator on $H$.
The aim of this section is to classify all maximally dissipative extensions of the \emph{minimal Laplacian} on the interval $[0,1]$
\begin{align}
   \label{eq:MinS}
   S:\qquad &
   \begin{aligned}[t]
      \mD(S) &= H_0^2(0,1) = \{ f\in H^2(0,1) : f(0) = f(1) = f'(0) = f'(1) = 0\},
      \\
      Sf &= -f''
   \end{aligned}
   \intertext{and of}
   \label{eq:MinA}
   A:\qquad &
   \begin{aligned}[t]
      \mD(A) &= \mD(S),
      \\
      Af &= (S + \I V)f = -f'' + \I V
   \end{aligned}
\end{align}
 using \eqref{eq:extensionsBT} to describe the extensions.
Clearly, $S$ is symmetric with adjoint operator
\begin{alignat*}{4}
   S^* &:\qquad 
   \mD(S^*) &&= H^2(0,1),\quad 
   &Sf &= -f''
   \intertext{and}
   A^* &:\qquad 
   \mD(A^*) &&= \mD(S^*) = H^2(0,1), \quad
   &A^*f &= (S^* - \I V)f = -f'' - \I Vf
\end{alignat*}
and its defect indices are $\eta_+(S) = \eta_-(S) = 2$.
Therefore, every maximally dissipative extension $\widehat S$ is a two-dimensional extension of $S$ and at the same time a two-dimensional restriction of $S^*$.

Clearly, $(\Gamma_0, \Gamma_1, \C^2)$ is a boundary triple for $S^*$ with 
\begin{equation}
   \label{eq:BoundaryTripleSstar}
   \Gamma_0: \mD(S^*)\to\C^2,
   \quad
   \Gamma_0 u = 
   \begin{pmatrix}
      u(0) \\ u(1)
   \end{pmatrix},
   \qquad
   \Gamma_1: \mD(S^*)\to\C^2,
   \quad
   \Gamma_1 u = 
   \begin{pmatrix}
      u'(0) \\ -u'(1)
   \end{pmatrix}
\end{equation}
which can be easily verfied by integration by parts.

Since all maximally dissipative extensions of $S$ are two-dimensional restrictions of $S^*$, they are necessarily of the form $T_{B,C}$ where $B, C$ are $2\times 2$ matrices such that $\rk(B|-C) =  2$.
It depends on the matrices $B,C$ if the extension $T_{B,C}$ is in fact dissipative.
Lemma~\ref{lem:MaxDissCritBT} leads to the following characterization of all maximally dissipative extensions of $S$.
\begin{proposition}
   \label{prop:KurasovCriterion}
   All maximally dissipative extensions of $S$ in \eqref{eq:MinS} are of the form
   \begin{equation}
      \label{eq:ExtSW}
      S_{B,C}:\qquad
      \mD(S_{B,C}) = \left\{ u \in \mD(S^*) : B\Gamma_0 u = C \Gamma_1 u \right\},
      \quad
      S_{B,C}f = S^*f = -f''
   \end{equation}
   where $B,C$ are $2\times 2$ matrices such that $\rk(B|-C) = 2$ and $\im(BC^*)\ge 0$.
   Every operator of this form is a maximally dissipative extension of $S$.
\end{proposition}

The proposition above was already proved in \cite[Theorem 3.7]{KMN2025} for the more general situation of a quantum graph.

\subsection{Maximally dissipative extensions of $S$}  
\label{subsec:MaxDissS}
In this section we want to classify all dissipative extensions of $S$ explicitly in terms of boundary conditions.
This is convenient, e.g., for calculating their adjoint operator.
\smallskip

Recall that for every $f,g\in \mD(S^*)$, integration by parts gives 
\begin{align*}
   \scalar{f}{S^*g} - \scalar{S^*f}{g} 
   = \scalar{\Gamma_0 f}{\Gamma_1 g} - \scalar{\Gamma_1 f}{\Gamma_0 g}
   =  \overline{f(0)}g'(0) - \overline{f(1)}g'(1)
+ \overline{f'(1)} g(1) - \overline{f'(0)} g(0).
\end{align*}
In particular,
\begin{equation*}
   \scalar{f}{S^*f} - \scalar{S^*f}{f} 
   = 2\I \im\big( \overline{f(0)}f'(0) - \overline{f(1)}f'(1) \big).
\end{equation*}

First let us consider all maximally dissipative operators between $S$ and $S^*$ as \emph{extensions} of $S$.

\begin{proposition}
   \label{prop:MaxDisExtS}
   An extension of $S$ is a maximally dissipative extension of $S$ if and only if it is of the form
   \begin{equation*}
      S_{\mW}:\qquad
      \mD( S_{\mW} ) = \mD(S) \dot + \mW,\quad
      Sf = S^*f = -f'',
   \end{equation*}
   where $\mW\subseteq \mD(S^*)//\mD(S)$ with $\dim \mW = 2$ and 
   \begin{equation}
      \label{eq:MaxDisExtS}
      \im
      \left(
      \begin{pmatrix}
	 u(0) & v(0) \\
	 u(1) & v(1)
      \end{pmatrix}^*
      \begin{pmatrix}
	  u'(0) &  v'(0) \\
	 -u'(1) & -v'(1) 
      \end{pmatrix}
      \right)
      \ge 0
   \end{equation}
   for one, and hence for every, basis $\{u,v\}$ of $\mW$.
   The extension is selfadjoint if and only if equality holds in \eqref{eq:MaxDisExtS} for one, and hence for every, basis of $\mW$.
\end{proposition}
\begin{proof} 
   Observe that $T$ is a maximally dissipative extension of $S$ if and only if $T$ is a two-dimensional extension of $S$ with $\mD(T) = \mD(S) \dot + \mW \subseteq \mD(S^*)$ and $\im\scalar{w}{S^*w} \ge 0$ for every $w\in\mW$.
   Therefore, there is a bijection between the set of all maximally dissipative extensions of $S$ and the set all subspaces
   \begin{equation}
      \label{eq:MaxDisExtW}
      \mW \subseteq \mD(S^*) // \mD(S)
      \quad\text{such that}\quad
      \dim(\mW) = 2
      \quad\text{and}\quad
      \im\scalar{w}{S^*w} \ge 0
      \text{ for all } w\in \mW.
   \end{equation}

   Lemma~\ref{rem:GreenStuffTBC}~\ref{item:GreenStuffTBC:i} shows that 
   $\im\scalar{\Gamma_0 w}{\Gamma_1 w} \ge 0$ for every $w\in \mW$ if and only if \eqref{eq:MaxDisExtS} holds.
   Hence $\mW$ satisfies \eqref{eq:MaxDisExtW}, and therefore corresponds to a maximally dissipative extension of $S$, if and only if \eqref{eq:MaxDisExtS} holds for one (and hence for every) basis of $\mW$.
\end{proof}

Proposition~\ref{prop:MaxDisExtS} describes the maximally dissipative extensions of $S$ by \emph{extending} its domain. 
Alternatively, we can describe all maximally dissipative extensions of $S$ as two-dimensional \emph{restrictions} of $S^*$ using Proposition~\ref{prop:KurasovCriterion}.
Recall that $\mD(S^*) = H^2(0,1)$.

\begin{theorem}
   \label{thm:MaxDisRestS}
   Every maximally dissipative extension $\widehat S$ of $S$ belongs to exactly one of the cases listed below, and we can write $\mD(\widehat S) = \mD(S)\dot + \gen\{u,\, v\}$ with $u,v\in\mD(S^*)$ which are linearly independent over $\mD(S)$.

   \begin{enumerate}[label={\upshape(\roman*)}]
      \item 
      \label{item:caseIS}
      $\mD(\widehat S)= 
      \left\{ f\in H^2(0,1) \,:\,
      \begin{pmatrix}
	 f(0) \\ f(1)
      \end{pmatrix}
      = 
      \begin{pmatrix}
	 c_{11} & c_{12} \\
	 c_{21} & c_{22}
      \end{pmatrix}
      \begin{pmatrix}
	 f'(0) \\ -f'(1)
      \end{pmatrix}
      \right\}$
      with
      $\im 
      \begin{pmatrix}
	 c_{11} & c_{12} \\
	 c_{21} & c_{22}
      \end{pmatrix}^*
      \ge 0$.
      \smallskip

      Explicitly: 
      $\im c_{11} \le 0,\ \im c_{22} \le 0$ and 
      $\im(c_{11}) \im(c_{22}) \ge \frac{1}{4} |c_{12} - c_{21}|^2$.
      \smallskip

      The functions $u,v$ can be chosen such that
      \begin{equation*}
      ( u(0),\, u(1),\, u'(0),\, u'(1) ) = (c_{11},\, c_{21}, 1, 0)
      \quad \text{and} \quad
      ( v(0),\, v(1),\, v'(0),\, v'(1) ) = (-c_{12},\, -c_{22}, 0, 1).
      \end{equation*}

      \item
      \label{item:caseIIS}
      $\mD(\widehat S)= 
      \left\{ f\in H^2(0,1)  \,:\,
      \begin{aligned}
	 f(1) - \overline{c_{22}} f(0) &= \overline{c_{22}} c_{12} f'(1) ,\\
	 f'(0) - c_{22} f'(1) &= 0 \\
      \end{aligned}
      \right\}$,
      \quad
      $c_{22}\neq 0$, 
      $\im (c_{12}\overline{c_{22}}) \ge 0$.

      The functions $u,v$ can be chosen such that
      \begin{equation*}
      ( u(0),\, u(1),\, u'(0),\, u'(1) ) = (\overline{c_{22}}^{-1},\, 1,\, 0,\, 0)
      \quad \text{and} \quad
      ( v(0),\, v(1),\, v'(0),\, v'(1) ) = (-c_{12},\, 0,\, c_{22},\, 1).
      \end{equation*}

      \item
      \label{item:caseIIIS}
      $\mD(\widehat S)= 
      \left\{ f\in H^2(0,1) \,:\,
      \begin{aligned}
	 f(0)  &= c_{11} f'(0),\\
	 f'(1) &= 0
      \end{aligned}
      \right\}$,
      \quad
      $\im (\overline{c_{11}}) \ge 0$.

      The functions $u,v$ can be chosen such that
      \begin{equation*}
      ( u(0),\, u(1),\, u'(0),\, u'(1) ) = (0,\, 1,\, 0,\, 0)
      \quad \text{and} \quad
      ( v(0),\, v(1),\, v'(0),\, v'(1) ) = (c_{11},\, 0,\, 1,\, 0).
      \end{equation*}

      \item
      \label{item:caseIVS}
      $\mD(\widehat S)= 
      \left\{ f\in H^2(0,1) \,:\,
      \begin{aligned}
	 f(1)  &= -c_{12} f'(1),\\
	 f'(0) &= 0
      \end{aligned}
      \right\}$,
      \quad
      $\im ( \overline{c_{12}}) \ge 0$.

      The functions $u,v$ can be chosen such that
      \begin{equation*}
      ( u(0),\, u(1),\, u'(0),\, u'(1) ) = (1,\, 0,\, 0,\, 0)
      \quad \text{and} \quad
      ( v(0),\, v(1),\, v'(0),\, v'(1) ) = (0,\, -c_{12},\, 0,\, 1).
      \end{equation*}

      \item
      \label{item:caseVS}
      $\mD(\widehat S)= 
      \left\{ f\in H^2(0,1) \,:\,
	 f'(0) = f'(1) = 0
      \right\}$.

      The functions $u,v$ can be chosen such that
      \begin{equation*}
      ( u(0),\, u(1),\, u'(0),\, u'(1) ) = (1,\, 0,\, 0,\, 0)
      \quad \text{and} \quad
      ( v(0),\, v(1),\, v'(0),\, v'(1) ) = (0,\, 1,\, 0,\, 0).
      \end{equation*}

   \end{enumerate}

   \noindent
   In all cases, $\widehat S f = S^*f = -f''$ for all $f\in \mD(\widehat S)$.
   Moreover, every operator of this form is a maximally dissipative extension of $S$.

\end{theorem}
\begin{proof}
   By Proposition~\ref{prop:KurasovCriterion} we only need to find, up to left multiplication by an invertible matrix $E$, all $2\times 4$ matrices $(B|-C)$ of full rank such that $\im(BC^*)\ge 0$.
   Clearly, all such matrices belong to exactly one of the following cases (up to multiplication from the left by an invertible $2\times 2$ matrix $E$), which correspond exactly to the cases in the claim of the theorem:

   \begin{enumerate}[label={\upshape(\roman*)}]

      \item 
      $(B|-C) = 
      \begin{pmatrix}
	 1 & 0 &\aug & -c_{11} & -c_{12} \\
	 0 & 1 &\aug & -c_{21} & -c_{22}
      \end{pmatrix}
      $.
      The dissipativity condition is $\im C^*  = -\im C \ge 0$.

      \item 
      $(B|-C) = 
      \begin{pmatrix}
	 1 & b_{12} &\aug & 0 & -c_{12} \\
	 0 & 0      &\aug &-1 & -c_{22}
      \end{pmatrix}
      $.
      Therefore
      $BC^* 
      = \begin{pmatrix}
	 b_{12} \overline{c_{12}} & 1 + b_{12} \overline{ c_{22} }\\
	  0 & 0
      \end{pmatrix}
      $
      and the extension is dissipative if and only if 
      $\im (b_{12} \overline{ c_{12} }) \ge 0$ and $1+ b_{12}\overline{ c_{22} } = 0$.
      This is the case if and only if 
      $b_{12}\neq 0$, $\overline{c_{22}} = -1/b_{12}$ and 
      $\im (b_{12} \overline{ c_{12} }) \ge 0$.

      \item 
      $(B|-C) = 
      \begin{pmatrix}
	 1 & b_{12} &\aug & -c_{11} & 0 \\
	 0 & 0      &\aug & 0 & -1
      \end{pmatrix}
      $.
      Therefore
      $BC^* 
      = \begin{pmatrix}
	 \overline{c_{11}} & b_{12} \\
	  0 & 0
      \end{pmatrix}
      $
      and the extension is dissipative if and only if 
      $\im (\overline{ c_{11} }) \ge 0$ and $b_{12} = 0$.

      \item 
      $(B|-C) = 
      \begin{pmatrix}
	 0 & 1 &\aug &  0 & -c_{12} \\
	 0 & 0 &\aug & -1 & -c_{22}
      \end{pmatrix}
      $.
      Therefore
      $BC^* 
      = \begin{pmatrix}
	 \overline{c_{12}} & \overline{ c_{22} }\\
	  0 & 0
      \end{pmatrix}
      $
      and the extension is dissipative if and only if 
      $\im (\overline{ c_{12} }) \ge 0$ and $c_{22} = 0$.

      \item 
      $(B|-C) = 
      \begin{pmatrix}
	 0 & 0 &\aug & -1 & 0 \\
	 0 & 0 &\aug & 0 & -1
      \end{pmatrix}
      $.
      Therefore
      $BC^* = 0$ which clearly has non-negative imaginary part.

      \item 
      $(B|-C) = 
      \begin{pmatrix}
	 0 & 1 &\aug & -c_{11} & 0 \\
	 0 & 0 &\aug & 0 & -1
      \end{pmatrix}
      $.
      Therefore
      $BC^* 
      = \begin{pmatrix}
         0 & 1 \\
         0 & 0 
      \end{pmatrix}
      \begin{pmatrix}
          \overline{c_{11}} & 0 \\
          0 & 1
      \end{pmatrix}
      = \begin{pmatrix}
	 0 & 1 \\
	  0 & 0
      \end{pmatrix}
      $,
      hence
      $\im(BC^*) = \frac{1}{2\I}
      \begin{pmatrix} 0 & 1 \\ -1 & 0
      \end{pmatrix}
      $
      which has eigenvalues $\pm \frac{1}{2}$,
      therefore this pair of matrices does not lead to a dissipative extension of $S$.
      \qedhere

   \end{enumerate}
\end{proof}

\begin{remark*}
   Clearly, the classification in the preceding theorem is somewhat arbitrary.
   We could also have grouped the extensions e.g. according to \\
   $(B|-C) = 
   \begin{pmatrix}
      b_{11} & b_{12} &\aug & -1 & 0 \\
      b_{21} & b_{22} &\aug & 0 & -1 
   \end{pmatrix},\
   \begin{pmatrix}
      b_{11} & 1 &\aug & 0 & 0 \\
      b_{21} & 0 &\aug & -c_{21} & -1 
   \end{pmatrix},\
   \begin{pmatrix}
      0 & 1 &\aug & 0 & 0 \\
      b_{21} & b_{22} &\aug & -c_{21} & -1 
   \end{pmatrix},\
   $ etc.
\end{remark*}

\subsection{Maximally dissipative extensions of $A=S+\I V$}  
\label{subsec:MaxDissA}

In this section we describe all maximally dissipative extensions of $A$.
Recall that they do not need to be restrictions of $S^* + \I V$.
First, we describe all maximally dissipative extensions in analogy to Proposition~\ref{prop:MaxDisExtS}.

\begin{theorem}
   \label{thm:MaxDisExtA}
   An extension of $A$ is a maximally dissipative extension of $A$ if and only if it is of the form
   \begin{equation*}
      A_{\mW, \mL}:\qquad
      \mD( A_{\mW,\mL} ) = \mD(S) \dot + \mW,\quad
      A_{\mW, \mL}(f_S+w) = (S^* + \I V)(f_S +w) + \mL w\:,
   \end{equation*}
   where $f_S\in\mD(S), w\in\mW$. Here, $\mW\subseteq \mD(S^*)//\mD(S)$ with $\dim \mW = 2$ and $\mL:\mW\to \range( V^{1/2} )$ 
   such that 
   \begin{equation}
      \label{eq:MaxDisExtA}
      \im
      \left(
      \begin{pmatrix}
	 u(0) & v(0) \\
	 u(1) & v(1)
      \end{pmatrix}^*
      \begin{pmatrix}
	  u'(0) &  v'(0) \\
	 -u'(1) & -v'(1) 
      \end{pmatrix}
      \right)
      \ge 
      \frac{1}{4} 
      \begin{pmatrix} 
	 \| V^{-1/2}\mL u\|^2  & \scalar{V^{-1/2}\mL u}{V^{-1/2}\mL v}
	 \\
	 \scalar{V^{-1/2}\mL v}{V^{-1/2}\mL u} & \| V^{-1/2}\mL v\|^2 
      \end{pmatrix}
   \end{equation}
   for one, and hence for every, basis $\{u, v\}$ of $\mW$. 
\end{theorem}
\begin{proof} 
   By Theorem~\ref{thm:ChristophMaxDiss} and Corollary~\ref{cor:domains}, the maximally dissipative extensions of $A$ are exactly the operators of the form 
   \begin{equation*}
      A_{\mW,\mL}:\qquad
      \begin{aligned}[t]
	 \mD( A_{\mW,\mL} ) &= \mD(A) \dot + \mW = \mD(S) \dot + \mW, 
	 \\\
	 A_{\mW,\mL} (f_S+w) &= (S^* + \I V)(f_S+w) + \mL w,
	 \quad f_S\in\mD(S), w\in\mW
      \end{aligned}
   \end{equation*}
   where $\mW\subseteq \mD(S^*)//\mD(S)$ has dimension 2 and $\mL:\mW\to \range(V^{1/2})$ such that \eqref{eq:basiccrit} holds.
   So we only need to show that \eqref{eq:MaxDisExtA} and \eqref{eq:basiccrit} are equivalent.
   This follows from Proposition~\ref{prop:dissCond}.
\end{proof}

Next we employ Theorem~\ref{thm:MaxDisRestS} together with the dissipativity criterion \eqref{eq:basiccrit} to find an explicit description of all maximally dissipative extensions of $A$.

\begin{corollary}
   \label{cor:MaxDisRestA}
   Every maximally dissipative extension $\widehat A$ of $A$ belongs to exactly one of the cases listed below.
   Moreover, every operator in these cases is a maximally dissipative extension of $A$.

   \begin{enumerate}[label={\upshape(\roman*)}]
      \item 
      $\begin{aligned}[t]
	 \mD(\widehat A)
	 &= 
	 \left\{ f\in H^2(0,1) \,:\,
	 \begin{pmatrix}
	    f(0) \\ f(1)
	 \end{pmatrix}
	 = 
	 \begin{pmatrix}
	    c_{11} & c_{12} \\
	    c_{21} & c_{22}
	 \end{pmatrix}
	 \begin{pmatrix}
	    f'(0) \\ -f'(1)
	 \end{pmatrix}
	 \right\},
	 \\
	 \widehat Af &= (S^* + \I V)f + f'(0)k_1 + f'(1)k_2
      \end{aligned} $
      \smallskip

      with $k_1, k_2\in\range(V^{1/2})$ and $c_{jk}\in \C$ such that 
      \smallskip

      $\im \begin{pmatrix}
	 c_{11} & c_{12} \\
	 c_{21} & c_{22}
      \end{pmatrix}^*
      \ge 
      \frac{1}{4} 
      \begin{pmatrix}
	 \|V^{-1/2} k_1 \|^2 & \scalar{V^{-1/2}k_1}{V^{-1/2}k_2} \\
	 \scalar{V^{-1/2}k_2}{V^{-1/2}k_1} & \|V^{-1/2} k_2 \|^2
      \end{pmatrix} $.

      \item
      \label{item:caseIIA}
      $\begin{aligned}[t]
	 \mD(\widehat A)
	 &= \left\{ f\in H^2(0,1)  \,:\,
	 \begin{aligned}
	 f(1) - \overline{c_{22}} f(0) &= \overline{c_{22}} c_{12} f'(1) ,\\
	 f'(0) - c_{22} f'(1) &= 0 \\
	 \end{aligned}
	 \right\},
	 \quad
	 \widehat Af &= (S^* + \I V)f + f(1)k_2
      \end{aligned} $
      \medskip

      with $k_2\in\range(V^{1/2})$ and $c_{12},\, c_{22}\in \C$ such that 
      $c_{22}\neq 0$, 
      $\im (c_{12} \overline{c_{22}}) \ge \frac{1}{4} \| V^{-1/2}k_2\|^2$.

      \item
      $\begin{aligned}[t]
	 \mD(\widehat A)
	 &= 
	 \left\{ f\in H^2(0,1) \,:\,
	 \begin{aligned}
	    f(0)  &= c_{11} f'(0),\\
	    f'(1) &= 0
	 \end{aligned}
	 \right\},
	 \quad
	 \widehat Af &= (S^* + \I V)f + f'(0)k_2
      \end{aligned} $
      \medskip

      with $k_2\in\range(V^{1/2})$ and $c_{11}$ such that 
      $\im (c_{11}^*) \ge \frac{1}{4} \|V^{-1/2} k_2 \|^2$.

      \item
      $ \begin{aligned}[t]
	 \mD(\widehat A)
	 &= 
	 \left\{ f\in H^2(0,1) \,:\,
	 \begin{aligned}
	    f(1)  &= -c_{12} f'(1),\\
	    f'(0) &= 0
	 \end{aligned}
	 \right\},
	 \quad
	 \widehat Af &= (S^* + \I V)f + f(1)k_2
      \end{aligned} $
      \smallskip

      with $k_2\in\range(V^{1/2})$ and $c_{12}\in \C$ such that 
      $\im (\overline{c_{12}}) \ge \frac{1}{4} \| V^{-1/2}k_2\|^2$.

      \item
      $ \begin{aligned}[t]
	 \mD(\widehat A)
	 &=
	 \left\{ f\in H^2(0,1) \,:\,
	 f'(0) = f'(1) = 0
	 \right\},
	 \quad
	 \widehat Af &= (S^* + \I V)f.
      \end{aligned} $

   \end{enumerate}

\end{corollary}
\begin{proof}
   By Corollary~\ref{cor:domains} we know that the domain of a maximally dissipative extension of $A$ coincides with the domain of a maximally dissipative extension $\widehat S$ of $S$. 
   Therefore, by Theorem~\ref{thm:MaxDisExtA}, given a domain as in Theorem~\ref{thm:MaxDisRestS}, we only need to identify all linear maps $\mL:\mW\to \range(V^{1/2})$ for which the dissipativity condition \eqref{eq:MaxDisExtA} holds.
   \smallskip

   The cases in the following list correspond to the cases in Theorem~\ref{thm:MaxDisRestS}, and the bases $u,v$ for $\mW$ are those given there.

   \begin{enumerate}[label={\upshape(\roman*)}]

      \item 
      Let $f\in \mD(\widehat A)$.
      Then clearly $f = f_S + f'(0)u + f'(1)v$ for some $f_S$ in $\mD(S)$ and
      $\widehat Af = (S^*+\I V)f + f'(0)\mL u + f'(1) \mL v$.
      The conditions on $k_1:= \mL u$, $k_2:= \mL v$ and the $c_{jk}$ follow directly from the dissipativity condition \eqref{eq:MaxDisExtA}.

      \item 
      Let $f\in \mD(\widehat A)$.
      Then $f = f_S + f(1)u + f'(1)v$ for some $f_S$ in $\mD(S)$ and
      $\widehat Af = (S^*+\I V)f + f(1)\mL u + f'(1) \mL v$.
      With $k_1:= \mL u$ and $k_2:= \mL v$, the dissipativity condition \eqref{eq:MaxDisExtA} gives
      \begin{equation*}
	 \im
	 \begin{pmatrix}
	    0 & 0 \\
	    0 & -\overline{c_{12}} c_{22} 
	 \end{pmatrix}
	 - \frac{1}{4} 
	 \begin{pmatrix}
	    \|V^{-1/2} k_1\|^2  & \scalar{V^{-1/2}k_1}{V^{-1/2}k_2} \\
	    \scalar{V^{-1/2}k_2}{V^{-1/2}k_1} & \|V^{-1/2} k_2\|^2 
	 \end{pmatrix}
	 \ge 0.
      \end{equation*}
      This holds if and only if 
      $\im (c_{12} \overline{c_{22}}) - \frac{1}{4} \| V^{-1/2}k_2\|^2 \ge 0$ and 
      $\| V^{-1/2}k_1\| = 0$.
      Since $V^{-1/2}$ is injective, the latter is true if and only if $k_1=0$.\footnote{Observe that $u\in\Hsym(\widehat S)$, see Section~\ref{subsec:5:2}}

      \item 
      Let $f\in \mD(\widehat A)$.
      Then $f = f_S + f(1)u + f'(0)v$ for some $f_S$ in $\mD(S)$ and
      $\widehat Af = (S^*+\I V)f + f(1)\mL u + f'(0) \mL v$.
      With $k_1:= \mL u$ and $k_2:= \mL v$, the dissipativity condition \eqref{eq:MaxDisExtA} gives
      \begin{equation*}
	 \im
	 \begin{pmatrix}
	    0 & 0 \\
	    0 & \overline{c_{11}}
	 \end{pmatrix}
	 - \frac{1}{4} 
	 \begin{pmatrix}
	    \|V^{-1/2} k_1\|^2  & \scalar{V^{-1/2}k_1}{V^{-1/2}k_2} \\
	    \scalar{V^{-1/2}k_2}{V^{-1/2}k_1} & \|V^{-1/2} k_2\|^2 
	 \end{pmatrix}
	 \ge 0.
      \end{equation*}
      This holds if and only if 
      $\im (\overline{c_{11}}) - \frac{1}{4} \| V^{-1/2}k_2\|^2 \ge 0$ and 
      $\| V^{-1/2}k_1\| = 0$.
      Since $V^{-1/2}$ is injective, the latter is true if and only if $k_1=0$.

      \item 
      Let $f\in \mD(\widehat A)$.
      Then $f = f_S + f(0)u + f'(1)v$ for some $f_S$ in $\mD(S)$ and
      $\widehat Af = (S^*+\I V)f + f(0)\mL u + f'(1) \mL v$.
      With $k_1:= \mL u$ and $k_2:= \mL v$, the dissipativity condition \eqref{eq:MaxDisExtA} gives
      \begin{equation*}
	 \im
	 \begin{pmatrix}
	    0 & 0 \\
	    0 & \overline{c_{12}}
	 \end{pmatrix}
	 - \frac{1}{4} 
	 \begin{pmatrix}
	    \|V^{-1/2} k_1\|^2  & \scalar{V^{-1/2}k_1}{V^{-1/2}k_2} \\
	    \scalar{V^{-1/2}k_2}{V^{-1/2}k_1} & \|V^{-1/2} k_2\|^2 
	 \end{pmatrix}
	 \ge 0.
      \end{equation*}
      This holds if and only if 
      $\im (\overline{c_{12}}) - \frac{1}{4} \| V^{-1/2}k_2\|^2 \ge 0$ and 
      $\| V^{-1/2}k_1\| = 0$.
      Since $V^{-1/2}$ is injective, the latter is true if and only if $k_1=0$.

      \item 
      Let $f\in \mD(\widehat A)$.
      Then $f = f_S + f(0)u + f(1)v$ for some $f_S$ in $\mD(S)$ and
      $\widehat Af = (S^*+\I V)f + f(0)\mL u + f(1) \mL v$.
      With $k_1:= \mL u$ and $k_2:= \mL v$, the dissipativity condition \eqref{eq:MaxDisExtA} gives
      \begin{equation*}
	 - \frac{1}{4} 
	 \begin{pmatrix}
	    \|V^{-1/2} k_1\|^2  & \scalar{V^{-1/2}k_1}{V^{-1/2}k_2} \\
	    \scalar{V^{-1/2}k_2}{V^{-1/2}k_1} & \|V^{-1/2} k_2\|^2 0
	 \end{pmatrix}
	 \ge 0
      \end{equation*}
      which holds if and only if 
      $k_1=k_2=0$.
      \qedhere

   \end{enumerate}
\end{proof}

\begin{remark}
   It should be noted that the range of $\mL$ can be two-dimensional only if both $B$ and $C$ are invertible.
   Moreover, if the domain of an extension $\widehat A$ coincides with the domain of a selfadjoint extension of $S$, then $\mL = 0$.
\end{remark}

%


\section{Complete non-selfadjointness} 
\label{sec:cnsa}

Let $T$ be a densely defined linear operator on a Hilbert space $H$. 
If $\mM$ is a subspace of $H$, we write $T_{\mM}$ for the restriction of $T$ to $\mM\cap \mD(T)$.
A closed subspace $\mM$ of $H$ is an \define{invariant subspace of } $T$ if $T(\mM\cap \mD(T)) \subseteq \mM$.
A subspace $\mM$ is called a \define{reducing subspace of} $T$ if 
\begin{align*}
   \mD(T) = ( \mD(T)\cap \mM ) \oplus ( \mD(T)\cap \mM^\perp )
\end{align*}
and both $\mM$ and $\mM^\perp$ are $T$-invariant.
In that case,
$\overline{\mD(T)\cap\mM} = \mM$ and $\overline{\mD(T)\cap\mM^\perp} = \mM^\perp$ and
we denote by $T_\mM$ the restriction of $T$ to $\mM$. 

If $\mM$ is a reducing subspace for $T$, then it is also a reducing subspace for $(T-\lambda)^{-1}$ for every $\lambda\in\rho(T)$ and for $T^*$ and 
$(T_{\mM})^* = (T^*)_{\mM}$.
\smallskip

A closed subspace $\mM$ is called a \define{selfadjoint subspace of} 
$T$ if it is a reducing subspace of $T$ and $T_\mM$ is selfadjoint.
The operator $T$ is called \define{completely non-selfadjoint} (cnsa) if its only selfadjoint subspace it the trivial space $\{ 0 \}$.
\smallskip

In the case of a maximally dissipative operator $T$ with $\mD(T)\subseteq 
\mD(T^*)$, Langer's Zerspaltungssatz \cite{LangerZerspaltung, Naboko1980}
says that there exists a  unique decomposition
   \begin{equation*}
      H = H_{sa}(T) \oplus H_{cnsa}(T)
   \end{equation*}
where $H_{sa}(T)$ and $H_{cnsa}(T)$ are closed and mutually orthogonal reducing subspaces of $T$ and 
$T_{sa} := T_{H_{sa}}$ is selfadjoint and $T_{H_{cnsa}}$ is cnsa.
The space $H_{sa}(T)$ is maximal in the sense that it contains every reducing selfadjoint subspace for $T$.
\smallskip

If $S$ is a symmetric cnsa operator, then it is also called a \define{simple operator}, see  e.g. \cite{AkhiezerGlazman}.
For any closed symmetric operator $S$, there is a reducing subspace $H_{sa}(S)$ such that 
$S|_{H_{sa}(S)}$ is selfadjoint and 
$S|_{H_{sa}(S)^\perp}$ is cnsa.
This subspace is given by , see \cite{Krein1949}, \cite[Corollary 3.4.5.]{BehrndtHassideSnoo},
\begin{equation}
   \label{eq:HsaBHdS}
   H_{sa}(S) 
   = \bigcap_{\lambda\in\C\setminus\R} \range (S-\lambda)
   = \Big( \gen_{\lambda\in\C\setminus\R}\{ \ker (S^*-\lambda) \}\Big)^\perp.
\end{equation}
In particular, $S$ is completely non-selfadjoint if and only if
\begin{equation*}
   \overline{\gen}_{ \lambda\in\C\setminus\R}\{ \ker (S^*-\lambda) \} = H.
\end{equation*}

\begin{remark}
   \label{prop:BHdS}
   Let $S$ be a closed symmetric operator on a Hilbert space $H$ and let $\Omega^\pm$ be subsets of $\C^\pm$ with an accumulation point in $\C^+$ and $\C^-$ respectively. 
   Then 
   \begin{equation}
      H_{sa}(S) 
      = \Big( 
	 \gen_{\lambda\in\C^+}\{ \ker (S^*-\lambda) \}
      + 
	 \gen_{\lambda\in\Omega^+}\{ \ker (S^*-\lambda) \}
      \Big)^\perp.
   \end{equation}
\end{remark}
\begin{proof}
   The claim follows from \eqref{eq:HsaBHdS} and 
   {\cite[Proposition 1.6.11.]{BehrndtHassideSnoo}} where it is shown that
   \begin{equation*}
      \begin{aligned}
	 \overline{\gen}_{\lambda\in\C^+}\{ \ker (S^*-\lambda) \}
	 &=
	 \overline{\gen}_{\lambda\in\Omega^+}\{ \ker (S^*-\lambda) \},
	 \\
	 \overline{\gen}_{\lambda\in\C^-}\{ \ker (S^*-\lambda) \}
	 &=
	 \overline{\gen}_{\lambda\in\Omega^-}\{ \ker (S^*-\lambda) \}.
      \end{aligned}
      \qedhere
   \end{equation*}
\end{proof}
\medskip

Moreover, if $T=S+\I V$ with a selfadjoint operator $S$ and a bounded operator $V\ge 0$, then we have the following formula for $H_{cnsa}(T)$:
\begin{equation*}
   H_{cnsa}(T) 
   = \overline{\gen}_{\lambda \in \C\setminus\R} \{ (S-\lambda)^{-1}(\range V) \}
   = \overline{\gen}_{\lambda \notin \sigma(T)\cup\R } \{ (T-\lambda)^{-1}(\range V) \}.
\end{equation*}
\medskip

The next remark will be used in Lemma~\ref{lem:widetildeScnsa}.

\begin{remark}[\protect{\cite[Proposition 3.4.8]{BehrndtHassideSnoo}}]
\label{prop:BHdS348}
   Let $S$ be a closed symmetric operator and assume that it admits a selfadjoint extension $\widehat S$ with $\sigma(\widehat S) = \sigma_p(\widehat S)$.
   Then $S$ is cnsa if and only if $\sigma_p(S) = \emptyset$.
\end{remark}

Let us return to maximally dissipative operators. 
If $\widehat A$ is a maximally dissipative operator with 
$\sigma(\widehat A)\cap \R \subseteq \sigma_p(\widehat A)$, then
\begin{equation}
   \label{prop:KMNRealPointSpec}
   H_{sa}(\widehat A) = \overline{\gen}\{ \ker (\widehat A - \lambda) : \lambda\in\R\}
\end{equation}
as shown in \cite[Proposition 2.6]{KMN2025}.
In particular, if $\sigma(\widehat A)\cap \R = \emptyset$, then $H_{sa}(\widehat A) = \{0\}$ and $\widehat A$ is cnsa.
\medskip

The symmetric space of a maximally dissipative operator is the largest subspace of $H$ where it coincides with its adjoint.

\begin{definition}
   \label{def:symmetricspace}
   Let $T$ be a closed densely defined linear operator. 
   Then its \define{symmetric subspace} is 
   \begin{equation*}
      \Hsym(T) := \ker(T - T^*),
   \end{equation*}
   where $T-T^*$ is defined on $\mD(T)\cap\mD(T^*)$.
\end{definition}

It was shown in \cite[Proposition~1.1]{Kuzhel1996} that
\begin{equation*}
   \Hsym(T) = \{ f\in\mD(T) : \scalar{Tf}{g} = \scalar{f}{Tg}\ \text{ for all }\ g\in\mD(T) \}.
\end{equation*}

\begin{remark}
   \begin{enumerate}[label={\upshape(\roman*)}]

      \item
      If $S$ is symmetric and closed, then $\Hsym(S) = \Hsym(S^*) = \mD(S)$.

      \item 
      For every maximally dissipative operator $\widehat A$ we have that 
      \begin{equation*}
	 H_{sa}(\widehat A) \cap \mD(\widehat A) \subseteq \Hsym(\widehat A).
      \end{equation*}

   \end{enumerate}
\end{remark}

Let $A= S + \I V$ where $S$ is a closed symmetric operator and $V$ is a bounded non-negative operator.
Then for every maximally dissipative extension $\widehat A$ of $A$ we have that $\mD(\widehat A)\subseteq \mD(S^*)$ and since $\widehat{A}^*\subseteq A^*$, we also have $\mD(\widehat A^*)\subseteq \mD(A^*)=\mD(S^*)$.

\begin{proposition}
   Let $A = S + \I V$ where $S$ is a closed symmetric operator and $V$ is a bounded non-negative operator.
   Then 
   \begin{enumerate}[label={\upshape(\roman*)}]

      \item\label{item:HsymMaxDis}
      $\mD(S) \cap \ker V \subseteq \Hsym(\widetilde A)$ 
      for any dissipative extension $\widetilde A$ of $A$,

      \item\label{item:HsymA}
      $\mD(S) \cap \ker V = \Hsym(A) \subseteq \mD(S)$.

   \end{enumerate}
\end{proposition}
\begin{proof}
   For the proof of \ref{item:HsymMaxDis} note that $S+\I V \subseteq \widetilde A$ implies $\widetilde A^* \subseteq S^* -\I V$.
   Let $f\in \mD(S) \cap \ker V$.
   Then $\widetilde A f = (S+\I V) f = Sf$ and $\widetilde A^* f = (S^* - \I V) f = S^*f = Sf$.
   Therefore $f\in\ker(\widetilde A - \widetilde A^*) = \Hsym(\widetilde A)$.

   For the claim \ref{item:HsymA}, take $f\in \Hsym(A)$.
   Then $f\in\mD(A)\cap\mD(A^*) = \mD(S)\cap\mD(S^*) = \mD(S)$ and
   \begin{equation*}
      0 = (A-A^*)f = (S + \I V) f - (S^* -\I V) f = 2\I Vf,
   \end{equation*}
   which proves that $\mD(S) \cap \ker V \supseteq \Hsym(A)$.
   The other inclusion follows from \ref{item:HsymMaxDis}.
\end{proof}

The next proposition is very useful for determining if a maximally dissipative operator is cnsa.

\begin{proposition}
   Let $A = S + \I V$ where $S$ is a closed symmetric operator and $V$ is a bounded non-negative operator.
   Let $\widehat A$ be a maximally dissipative extension of $A$ with 
   $\mD(S)  \cap \ker V = \Hsym(\widehat A)$.
   Then 
   \begin{equation*}
      H_{sa}(\widehat A) \perp H_{cnsa}(S).
   \end{equation*}

   In particular, if $S$ is cnsa and $\mD(S)  \cap \ker V = \Hsym(\widehat A)$, then $\widehat A$ is cnsa.
\end{proposition}
\begin{proof}
   Let $f\in H_{sa}(\widehat A)$. Then, since $H_{sa}(\widehat A)$ reduces $\widehat A$ and $\widehat A^*$, we have that
   $(\widehat A - \lambda)^{-1}f\in H_{sa}(\widehat A)$ for all $\lambda\in\rho(\widehat A)$ and 
   $(\widehat A^* - \lambda)^{-1}f\in H_{sa}(\widehat A)$ for all $\lambda\in\rho(\widehat A^*)$. In particular, $(\widehat A - \lambda)^{-1}f, (\widehat A^* - \lambda)^{-1}f \in\Hsym(\widehat{A})=\mD(S)\cap\ker(V)$ and thus, $\widehat{A}(\widehat A - \lambda)^{-1}f=\widehat{A}^*(\widehat A - \lambda)^{-1}f=S(\widehat A - \lambda)^{-1}f$\:.
   \smallskip

   \noindent
   Now, let $\lambda\in\C\setminus\R$ and let $\eta\in\ker(S^*-\lambda)$.
   If $\im\lambda > 0$, then $\overline\lambda\in\rho(\widehat A)$ and 
   \begin{equation*}
      \scalar{\eta}{f}
      = \scalar{\eta}{ (\widehat A-\overline\lambda) (\widehat A-\overline\lambda)^{-1} f }
      = \scalar{\eta}{ (S-\overline\lambda) (\widehat A-\overline\lambda)^{-1} f }
      = \scalar{( S^* - \lambda) \eta}{ (\widehat A-\overline\lambda)^{-1} f }
      = 0.
   \end{equation*}

   \noindent
   If on the other hand $\im\lambda < 0$, then $\overline\lambda\in\rho(\widehat A^*)$ and 
   \begin{equation*}
      \scalar{\eta}{f}
      = \scalar{\eta}{ (\widehat A^* - \overline\lambda) (\widehat A^* - \overline\lambda)^{-1} f }
      = \scalar{\eta}{ (S-\overline\lambda) (\widehat A^* -\overline\lambda)^{-1} f }
      = \scalar{( S^* - \lambda) \eta}{ (\widehat A^*- \overline\lambda)^{-1} f }
      = 0.
   \end{equation*}
   Hence $f\perp \overline{\gen}_{\lambda\in\C\setminus\R} \{ \ker (S^* - \lambda) \} = H_{cnsa}(S)$.
\end{proof}



\subsection{The symmetric subspace and boundary triples} 
\begin{theorem}
   \label{thm:SymSpaceBT}
   Let $T=T_{B,C}$ be an extension of a closed symmetric operator $S$ as in \eqref{eq:extensionsBT}.
   Then the symmetric space of $T$ is 
   \begin{align}
	 \label{eq:SymSpaceBT}
	 \begin{aligned}
        \Hsym(T) &= \mD(T)\cap\mD(T^*) 
	      \\
	    &= \left\{ g\in H : \exists\, \vec a  \in\ker( \im(BC^*) )\ \text{ s.t. }\:
	    \Gamma_0 g = C^* \vec a,\
	    \Gamma_1 g = B^* \vec a
	    \right\}.
	 \end{aligned}
    \end{align}
    In particular, if $\rk( \im(BC^*)) = \ell$, then $\Hsym(T)$ is an $\ell$-dimensional restriction of $\mD(T^*)$ and an $(\eta-\ell)$-dimensional extension of $\mD(S)$, i.e.
    $\dim(\mD(T^*)/\Hsym) = \ell$ and 
    $\dim(\Hsym/\mD(S)) = \eta-\ell$.
\end{theorem}
\begin{proof}
   Note that $\Hsym(T) = \mD(T)\cap \mD(T^*)$ consists of all $f\in \mD(T^*)$ such that 
   $\scalar{T^*g}{f} = \scalar{g}{T^*f}$ for all $g\in\mD(T^*)$.
   By Lemma~\ref{lem:AdjointBT} there exist
   $\vec a,\, \vec b \in \C^\eta$ such that 
   $\Gamma_0 f = C^* \vec a$, $\Gamma_1 f = B^* \vec a$
   and $\Gamma_0 g = C^* \vec b$, $\Gamma_1 g = B^* \vec b$.
   This shows that
   \begin{align*}
      0 &=  
      \scalar{T^*g}{f} - \scalar{g}{Tf}
      =  \scalar{S^* g}{f} - \scalar{g}{S^* f}
      =  \scalar{\Gamma_1 g}{\Gamma_0 f} - \scalar{\Gamma_0 g}{\Gamma_1 f}
      \\
      &=  \scalar{B^* \vec b}{C^* \vec a} - \scalar{C^*\vec b}{B^* \vec a}
      \\
      &= \scalar{\vec b}{ (BC^* - CB^*) \vec a}.
   \end{align*}
   Since this must hold for all $\vec b\in\C^\eta$, it follows that 
   $\vec a\in\ker(BC^*-CB^*) = \ker \im(BC^*)$.
\end{proof}

\subsection{Maximally dissipative extensions and complete non-selfadjointness}

Let us assume that $\im(BC^*) \ge 0$ and that $\rk( \im(BC^*)) = \ell$ and let $\widehat S$ be the maximally dissipative extension of $S$ associated to $B, C$ according to Theorem~\ref{thm:MaxDisRestS}.
Theorem~\ref{thm:SymSpaceBT} shows that $\Hsym(\widehat S)$ is an $\ell$-dimensional restriction of both $\mD(\widehat S)$ and $\mD(\widehat S^*)$.
Recall that the symmetric subspace of $\widehat S$ is
$\Hsym(\widehat S) = \ker(\widehat S - \widehat S^*)$.

\begin{definition}
    \label{def:symmetricpart}
    We call the operator 
   \begin{align*}
      \Ssym:\qquad \mD(\Ssym) = \Hsym(\widehat S),\quad \Ssym f = S^*
   \end{align*}
   the \define{symmetric part of} $S$.
\end{definition}

\begin{lemma}
   \label{lem:domainHsym}
   The domain of the largest symmetric restriction of $\widehat S$ is $\Hsym(\widehat S)$.
   The symmetric part $\Ssym$ of $\widehat S$
   is symmetric and 
   $\Ssym \subseteq \widehat S \subseteq \Ssym^*$ and 
   $\Ssym \subseteq \widehat S^* \subseteq \Ssym^*$.
   Their domains are
   \begin{align*}
      \mD(\widehat S) &= 
      \left\{ f\in \mD(S^*) : \
      B\Gamma_0 f - C\Gamma_1 f = 0
      \right\},
      \\[1ex]
      \mD(\widehat S^*) &= 
      \left\{ f\in \mD(S^*) : \exists\, \vec a \in \C^\eta\ \text{ s.t. }\: \
      \Gamma_0 f = C^*\vec a,\ \Gamma_1 f = B^*\vec a
      \right\},
      \\[1ex]
      \mD(\Ssym) &= 
      \left\{ f\in \mD(S^*) :\exists\, \vec a \in\ker(\im(BC^*))\ \text{ s.t. }\: \
      \Gamma_0 f = C^*\vec a,\ \Gamma_1 f = B^*\vec a      
      \right\},
      \\[1ex]
      \mD(\Ssym^*) &= 
      \left\{ f\in \mD(S^*) : \ B\Gamma_0 f - C\Gamma_1 f \in\range(\im (BC^*) ) \right\}.
   \end{align*}
\end{lemma}
\begin{proof}
   The claim for $\mD(\Ssym)$ is clear since both $\widehat S$ and $\widehat S^*$ are restrictions of $S^*$, so they coincide in the intersection of their domains.
   The formulas for the domains of $\widehat S^*$ and $\Ssym$ follow from 
   Lemma~\ref{lem:AdjointBT} and Theorem~\ref{thm:SymSpaceBT}.

   Recall that $\Ssym^*$ is a restriction of $S^*$.
   So in order to determine $\mD(\Ssym^*)$, we take $f\in\mD(\Ssym)$ and $\vec a\in \ker(\im(BC^*))$ such that $\Gamma_0 f = C^*\vec a$ and $\Gamma_1 f = B^*\vec a$.
   Then, for every $g\in\mD(S^*)$,
   \begin{align*}
      \scalar{\Ssym f}{g} - \scalar{f}{S^* g}
      &= \scalar{S^* f}{g} - \scalar{f}{S^* g}
      = \scalar{\Gamma_1 f}{\Gamma_0 g} - \scalar{\Gamma_0 f}{\Gamma_1 g}
      \\
      &= \scalar{B^*\vec a}{\Gamma_0 g} - \scalar{C^* \vec a}{\Gamma_1 g}
      = \scalar{\vec a}{(B\Gamma_0 - C\Gamma_1) g}.
   \end{align*}
   Hence, $g\in\mD(\Ssym^*)$ if and only if 
   $(B\Gamma_0 - C\Gamma_1) g 
   \in ( \ker(\im (BC^*)) )^\perp 
   = \range(\im (BC^*))$.
\end{proof}

Since $\Ssym$ is symmetric with $\dim( \mD(\Ssym^*)/\mD(\Ssym) ) = 2\ell$ 
we can choose a boundary triple  $(\wGamma_0,\, \wGamma_1, \C^{\ell})$  for $\Ssym^*$.
Then every $\ell$-dimensional extension $T$ of $\Ssym$ with $\Ssym \subseteq T\subseteq \Ssym^*$ is of the form
\begin{equation*}
   \Ssym[,\widetilde B, \widetilde C]: \qquad
   \mD(\Ssym[,\widetilde B, \widetilde C]) =
   \left\{ f\in \mD(S^*) : \widetilde B \wGamma_0 f = \widetilde C \wGamma_1 \right\},
   \quad \Ssym[, \widetilde B, \widetilde C] f = S^* f 
\end{equation*}
for $\wB, \wC \in M(\ell\times \ell)$ with $\rk(B|-C) = \ell$.
The extension is selfadjoint if and only if $\im(\widetilde B \widetilde C^*) = 0$ and 
it is dissipative if and only if $\im(\widetilde B \widetilde C^*) \ge 0$.
\smallskip

\begin{remark}
   The matrices $\widetilde B, \widetilde C$ that lead to the maximally dissipative operator $\widehat S = \Ssym[, \widetilde B, \widetilde C]$ are invertible, so we may assume that $B=\id$ and that $\widetilde C$ is invertible with $\im \widetilde C > 0$.
\end{remark}
\begin{proof}
   If either $\widetilde B$ or $\widetilde C$ is not invertible, then $\im(\wB\wC^*)$ is not invertible because $\ker(\wB \wC^*)\subseteq \ker(\im(\wB\wC^*))$. 
   But then there would be a symmetric operator $S'$ with $\Ssym \subsetneq S' \subseteq \widehat S$.
\end{proof}

In conclusion, by Lemma~\ref{lem:AdjointBT} we can write the domains of $\Ssym$, $\widehat S$ and $\widehat S^*$ as
\begin{align*}
   \mD(\Ssym)
   &= \{ f\in \mD(\Ssym^*) : \wGamma_0 f = \wGamma_1 f = 0 \},
   \\
   \mD(\widehat S)
   &= \{ f\in \mD(\Ssym^*) : \wB\wGamma_0 f = \wC \wGamma_1 f \}
   = \mD(\Ssym) \dot + \mU,
   \\
   \mD(\widehat S^*)
   &= \{ f\in \mD(\Ssym^*) : \wB^{*-1}\wGamma_0 f = \wC^{*-1} \wGamma_1 f \}
   \\
   &= \{ f\in \mD(\Ssym^*) : \wGamma_0 f = \wC^* \wB^*{}^{-1} \wGamma_1 f \}
\end{align*}
where $\mU = \gen\{ u_1,\,\dots, \, u_\ell \} \subseteq \mD(\Ssym^*)//\mD(\Ssym^*)$ with $\dim\mU = \ell$
and the $u_j$ can be chosen such that $\wGamma_1 u_j = \vec e_j\in\C^\ell$.

\begin{proposition}
\label{thm:widehatAadjoint}
   Let $\widehat{A}$ be a maximally dissipative extension of $S+\I V$ as described in Theorem~\ref{thm:AdjointOfA}. 
   Then, the operator $\widehat A^*$ is given by
   \begin{align*}
      \widehat A^*:\qquad
      \widehat A^*f = (S^* - \I V)f,
      \qquad
      \begin{aligned}[t]
	 \mD(\widehat A^*) 
	 &= \left\{ f\in \mD(\Ssym^*) : 
	 \wGamma_0 f =  \wC^* \wB^{*-1}\wGamma_1 f - 
     \begin{pmatrix}
         \scalar{k_1}{f} \\ \vdots \\ \scalar{k_\ell}{f} 
     \end{pmatrix}
	 \right\}.
      \end{aligned}
   \end{align*}

   \noindent
   The symmetric subspace of $\widehat A$ is
   \begin{align}
      \label{eq:ASymSpaceBT}
      \Hsym(\widehat A)
      &= \left\{ f\in \mD(\widehat A) : 
      \Big(\im( \wC^* \wB^{*-1})  - \frac{1}{4} M_{\mathtt K} \Big) \wGamma_1 f = 0,\
      2\I V f = (\wGamma_1 f)^t \mathtt K
      \right\}
   \end{align}
   where
   \begin{equation}
   \label{eq:MmathttK}
       M_{\mathtt K} =  \left(
         \lrscalar{V^{-1/2}k_r}{V^{-1/2}k_s}
         \right)_{r,s=1}^\ell
         \in M(\ell\times \ell).
   \end{equation}
   Moreover,
   \begin{align}
      \label{eq:ASymspaceVBT}
      \Hsym(\widehat A) \cap \ker V
      = \Hsym(\widehat S) \cap \ker V
      = \mD(\Ssym) \cap \ker V.
   \end{align}
\end{proposition}
\begin{proof}

   The claim about the domain of $\widehat A^*$ follows from Theorem~\ref{thm:AdjointOfA} with $\Ssym$ instead of $S$.
   
   Let $f\in\Hsym(\widehat A)$.
   The condition $f\in\mD(\widehat A)\cap\mD(\widehat A^*)$ is equivalent to
   \begin{equation}
   \label{eq:55}
      \wGamma_0 f = \wB^{-1}\wC\wGamma_1 f
      \qquad\text{and}\qquad
      2\I \im(\wC^* \wB^{*-1})\wGamma_1 f
      =
      \begin{pmatrix}
          \scalar{k_1}{f} \\ \vdots \\ \scalar{k_\ell}{f} 
      \end{pmatrix}.
   \end{equation}
   Next, $\widehat Af = \widehat A^*f$ gives that 
   $0 = 2\I Vf + \mL P f$.
   Therefore
   \begin{align*}
      \begin{pmatrix}
          \scalar{k_1}{f} \\ \vdots \\ \scalar{k_\ell}{f} 
      \end{pmatrix}
      &=
      \begin{pmatrix}
          \scalar{V^{-1/2}k_1}{V^{1/2}f} \\ \vdots \\ \scalar{V^{-1/2}k_\ell}{V^{1/2}f} 
      \end{pmatrix}
      =
      \frac{\I}{2}
      \begin{pmatrix}
          \scalar{V^{-1/2}k_1}{V^{-1/2}\mL P f} \\ \vdots \\ \scalar{V^{-1/2}k_\ell}{V^{-1/2}\mL P f} 
      \end{pmatrix}
      \\[2ex]
      &=
      \frac{\I}{2}
      \begin{pmatrix}
          \scalar{V^{-1/2}k_1}{V^{-1/2}(\wGamma_1 f)^t\mathtt K} \\ \vdots \\ \scalar{V^{-1/2}k_\ell}{V^{-1/2}(\wGamma_1 f)^t\mathtt K} 
      \end{pmatrix}
      =
   \frac{\I}{2}
      \begin{pmatrix}
          \scalar{V^{-1/2}k_1}{(\wGamma_1 f)^tV^{-1/2}\mathtt K} \\ \vdots \\ \scalar{V^{-1/2}k_\ell}{(\wGamma_1 f)^tV^{-1/2}\mathtt K} 
      \end{pmatrix}
      \\[2ex]
      &=
    \frac{\I}{2}
      \left(
      \lrscalar{V^{-1/2}k_r}{V^{-1/2}k_s}
      \right)_{r,s=1}^\ell
      \wGamma_1 f.
   \end{align*}
   Inserting this into \eqref{eq:55}, we obtain
   \begin{equation}
      0 =  2\I \im(\wC^* \wB^{*-1})\wGamma_1 f
      -
      \begin{pmatrix}
          \scalar{k_1}{f} \\ \vdots \\ \scalar{k_\ell}{f} 
      \end{pmatrix}
      = 2\I \left(
         \im(\wC^* \wB^{*-1})
         -
         \frac{1}{4}
         \left(
         \lrscalar{V^{-1/2}k_r}{V^{-1/2}k_s}
         \right)_{r,s=1}^\ell
         \right)
         \wGamma_1 f.
   \end{equation}
   This shows \eqref{eq:ASymSpaceBT}.
   
   For \eqref{eq:ASymspaceVBT} we use that $f\in \ker V$ implies $\scalar{k_j}{f} = 0$ since $k_j\in\range V^{1/2} \subseteq (\ker V)^\perp$.
   Moreover $0 = (\wGamma_1 f)^t\mathtt K$ implies that $\wGamma_1 f = 0$ since $k_1,\,\dots,\, k_\ell$ are linearly independent.
   Therefore
   \begin{align*}
      \Hsym(\widehat A) \cap \ker V
      &= \left\{ f\in \mD(\widehat A) : 
      \wGamma_1 f = 0,\
      Vf = 0
      \right\}
      \\
      &= \left\{ f\in \mD(\Ssym^*) : 
      \wB\wGamma_0 f = \wC\wGamma_1 f,\ 
      \wGamma_1 f = 0,\
      Vf = 0
      \right\}
      \\
      &= \left\{ f\in \mD(\Ssym^*) : 
      \wGamma_0 f = \wGamma_1 f = 0,\
      Vf = 0
      \right\}
      \\
      &= \mD(\Ssym) \cap \ker V
      = \Hsym(\widehat S) \cap \ker V.
      \qedhere
   \end{align*}
\end{proof}
\smallskip

Recall that we always have $\mD(\Ssym)\cap \ker V \subseteq \Hsym(\widehat A)$.
In the next lemma, we investigate how much the symmetric subspace of $\widehat A$ can differ from $\mD(\Ssym)\cap \ker V$.
\begin{lemma}
   \label{lem:W}
   There exists a subspace $W\subseteq \Hsym(\widehat A)$ with 
   $W\cap \mD(\Ssym) = W \cap \ker V = \{0\}$
   and $\dim W \le \nu := \ell - \rk( \im (\wC^* \wB^{*-1}) - \frac{1}{4} M_{\mathtt K})$ such that 
   \begin{align}
      \Hsym(\widehat A)
      = \big( \mD(\Ssym) \cap \ker V \big) \dot + W.
   \end{align}
   In particular, $W=\{0\}$ if $\im (\wC^* \wB^{*-1}) - \frac{1}{4} M_{\mathtt K} \geq\varepsilon I$ for some $\varepsilon>0$.
\end{lemma}
\begin{proof}
   By \eqref{eq:ASymSpaceBT} and \eqref{eq:ASymspaceVBT} we have that 
   \begin{align*}
      \Hsym(\widehat A)
      &= \big( \mD(\Ssym) \cap \ker V \big)
      \cup
      \left\{ f\in \mD(\widehat A) : 
      \wGamma_1 f \in \ker \Big( \im(\wC^* \wB^{*-1}) - \frac{1}{4} M_{\mathtt K} \Big),\
      2\I V f = (\wGamma_1 f)^t \mathtt K \neq 0
      \right\}
      \\
      &= \big( \mD(\Ssym) \cap \ker V \big)
      \dot + W
   \end{align*}
   where $W\subseteq \Hsym(\widehat A)$ has dimension at most $\dim\ker( \im(\wC^*\wB^{*-1}) - \frac{1}{4} M_{\mathtt K}) = \nu$.
\end{proof}

In Theorem~\ref{thm:dimleq1} we will see that the dimension of $W$ gives an upper bound for the dimension of the selfadjoint subspace of $\widehat A$ if $\Ssym$ is completely non-selfadjoint.
\medskip

In the \define{critical case} 
$\rk(\im (C^*B^{*-1}) - \frac{1}{4}M_{\mathtt K}) = 0$,
the dissipativity arising purely from the boundary condition (i.e., from the subspace $\mU$) balances exactly the dissipative effect resulting from the extra action of $\widehat A$ on $\mU$.
If 
$0 < \rk(\im (C^*B^{*-1}) - \frac{1}{4}M_{\mathtt K}) < \ell$, the effects are balanced only on some subspace of $\mU$, while in the \define{non-critical case} $\rk(\im (C^*B^{*-1}) - \frac{1}{4}M_{\mathtt K}) = \ell$
the dissipativity from the boundary conditions (i.e., from the subspace $\mU$), is stronger than the effect resulting from the extra action of $\widehat A$ on $\mU$ and  $H_{sa}(\widehat A) = H_{sa}(\Ssym)$  as the next proposition shows.

\begin{proposition}
   \label{prop:keineAhnung}
   \begin{enumerate}[label={\upshape(\roman*)}]

      \item
      If $G\subseteq \Hsym(\widehat A)\cap \ker V$ is a selfadjoint subspace for $\widehat A$, then $G\subseteq H_{sa}(\Ssym)$.

      \item
      If $\rk( \im(\wC^* \wB^{*-1}) - \frac{1}{4} M_{\mathtt K} ) = \ell$, 
      then 
      $H_{sa}(\widehat A)\subseteq H_{sa}(\Ssym)$.

   \end{enumerate}
\end{proposition}
\begin{proof}
   \begin{enumerate}[label={\upshape(\roman*)}]
      \item
      Let $G\subseteq \Hsym(\widehat A)\cap \ker V$ be a selfadjoint subspace of $\widehat A$.
      The operators $\widehat A$ and $\Ssym$ coincide on $\Hsym(\widehat A)\cap \ker V$, so it follows that $G$ is $\Ssym$-invariant and
      $\mD(\Ssym) \subseteq \mD(\widehat A) = (G \cap \mD(\widehat A)) \oplus (G^\perp \cap \mD(\widehat A))$
      implies $\mD(\Ssym) = (G \cap \mD(\Ssym)) \oplus (G^\perp \cap \mD(\Ssym))$.

      It remains to show that $\Ssym f \in G^\perp$ for every $f\in G^\perp \cap \mD(\Ssym)$.
      If $f\in G^\perp \cap \mD(\Ssym)$,
      then $\Ssym f = \widehat A f - \I V f - (\wGamma_1 f)^t \mathtt{K} \in G^\perp$,
      because $\widehat A f \in G^\perp$ since $G$ is reducing a subspace for $\widehat A$, and $\I V f, k_1,\dots, k_\ell \in (\ker V)^\perp \subseteq \Hsym(\widehat A)^\perp\subseteq G^\perp$.

      \item
      By hypothesis, $\ker( \im(\wC^* \wB^{*-1}) - \frac{1}{4} M_{\mathtt K} ) = \{0\}$, therefore $\Hsym(\widehat A) \cap \ker V = \Hsym(\widehat S)$.
      Since every selfadjoint subspace of $\widehat A$ is contained in $\Hsym(\widehat A)$, it follows that $H_{sa}(\widehat A) \subseteq H_{sa}(\Ssym)$.
      \qedhere
      
   \end{enumerate}

\end{proof}

\begin{proposition}
   \label{prop:resolventf}
   Assume that 
   $\Ssym h\in \ker V$ for all $h\in \mD(\Ssym)\cap \ker V$.
   If $f\in H_{sa}(\widehat A)$ is such that 
   $(\widehat A - \lambda)^{-1}f\in \mD(\Ssym)\cap \ker V$ 
   for some $\lambda\in\rho(\widehat A)$, 
   then $f\in H_{sa}(\Ssym)$.
\end{proposition}
\begin{proof}
   We will show that  $f\perp \ker(\Ssym^* - \mu)$ for every $\mu\in\C\setminus\R$ with $|\mu|$ large enough, say $|\mu| > C$ for some $C$.
   Then, Remark~\ref{prop:BHdS} shows that 
   $f \in \bigcap_{\mu\in \C\setminus\R, |\mu|> C} 
   \range(\Ssym^* - \mu) = H_{cnsa}(\Ssym)^\perp = H_{sa}(\Ssym)$.

   We set $g:= (\widehat A - \lambda)^{-1} f$.
   For every $\mu\in\C_-\setminus\{\lambda\}$ we have that $\mu\in\rho(\widehat A)$ and the resolvent identity gives
   \begin{align*}
      (\lambda - \mu) (\widehat A - \mu)^{-1} g
      &= 
      (\lambda - \mu) (\widehat A - \mu)^{-1} (\widehat A - \lambda)^{-1} f
      =
      (\widehat A - \lambda)^{-1} f - (\widehat A - \mu)^{-1} f
      \\
      &= g - (\widehat A - \mu)^{-1} f.
   \end{align*}
   Hence
   $(\widehat A - \mu)^{-1} g 
   \subseteq H_{sa}(\widehat A)\cap \mD(\widehat A)
   \subseteq \Hsym(\widehat A)
   = \big(\mD(\Ssym)\cap\ker V \big) \dot + W
   $
   where $W$ has dimension at most $\ell$, see 
   Lemma~\ref{lem:W}.
   Therefore we can write 
   $(\widehat A - \mu)^{-1} g = h + w$ with $h\in\mD(\Ssym)\cap \ker V$ and $w\in W$.
   It follows that
   \begin{align*}
      g = (\widehat A - \mu)(h+w)
      = \Ssym h - \mu h + \Ssym^* w + \I V w + \mL w - \mu w.
   \end{align*}
   By assumption, $g, h, \Ssym h \in \ker V$, hence it follows that 
   $\Ssym^* w + \I V w + \mL w - \mu w \in \ker V$.
   Therefore
   \begin{align*}
      0 &= | \scalar{Vw}{ \Ssym^* w + \I V w + \mL w - \mu w } |
      = | \scalar{Vw}{ (\Ssym^* + \mL ) w } + \I \scalar{Vw}{ V w } - \mu \scalar{Vw}{ w } |
      \\
      &\ge |\mu| \|V^{1/2}w\|^2 
      - \| Vw\| \| \Ssym^* w \|  - \| Vw\| \| \mL w \| - \|Vw\|^2.
   \end{align*}
   Since $\dim W$ is finite and $W\cap\ker V = \{0\}$, there exists $c > 0$ such that $\|V^{1/2} w \|^2 \ge c \|w\|^2$ for all $w\in W$.
   Moreover, since $\dim W<\infty$, there exists $c' < \infty$ such that 
   $\| Vw\| \| \Ssym^* w \| + \| Vw\| \| \mL w \| + \|Vw\|^2 \le c' \|w\|^2$ for all $w\in W$.
   Therefore
   \begin{align*}
      0 \ge (|\mu| c - c') \|w\|^2,
   \end{align*}
   so we conclude that $w=0$ if $|\mu| > c'/c =: C$ and consequently $g\in \mD(\Ssym)\cap\ker V$.
   Therefore, for $|\mu| > C$ and $\eta_\mu \in \ker(\Ssym^* - \overline{\mu})$, we find that 
   \begin{align*}
      \scalar{\eta_\mu}{f}
      = \scalar{\eta_\mu}{(\widehat A - \overline \mu)(\widehat A - \mu)^{-1}f}
      = \scalar{\eta_\mu}{(\Ssym - \overline \mu)(\widehat A - \mu)^{-1}f}
      = \scalar{(\Ssym^* - \mu)\eta_\mu}{(\widehat A - \mu)^{-1}f}
      = 0.
   \end{align*}
   The proof for $\mu\in \C_+$ is similar, using that $\mu\in\rho(\widehat A^*)$.
\end{proof}

If in addition to the hypotheses in Proposition~\ref{prop:resolventf} we assume that $\Ssym$ is cnsa, then we can conclude that the selfadjoint subspace of $\widehat A$ is at most $\ell$-dimensional.
(It cannot have a selfadjoint subspace contained in $\mD(\Ssym)\cap\ker V$ because that would lead to a selfadjoint subspace of $\Ssym$.)

\begin{theorem}
   \label{thm:dimleq1}
   Assume that $\Ssym h\in \ker V$ for all $h\in \mD(\Ssym)\cap \ker V$ 
   and assume that $\Ssym$ is cnsa.
   Let $W$ as in Lemma~\ref{lem:W}.
   Then $\dim( H_{sa}(\widehat A) ) \le \dim W$. 
\end{theorem}
\begin{corollary}
      If $W = \{0 \}$, then $H_{sa}(\widehat A) = \{0\}$.
\end{corollary}
\begin{proof}[Proof of Theorem~\ref{thm:dimleq1}]
   Let $\mu :=\dim W$ and 
   let $f_1,\, \dots,\, f_{\mu+1}\in H_{sa}(\widehat A) \subseteq \big( \mD(\Ssym) \cap \ker V \big) \dot + W$.
   Then  for every $\lambda\in\rho (\widehat A)$, we have 
   $(\widehat A - \lambda)^{-1}f_1,\, \dots,\, (\widehat A - \lambda)^{-1}f_{\mu+1}$ belong to
   $\big( \mD(\Ssym) \cap \ker V \big) \dot + W$.
   Since $\dim W = \mu$, there exists a linear combination 
   $h = c_1 f _1 + \dots + c_\mu f_{\mu+1}$ such that 
   $(\widehat A - \lambda)^{-1}h \in \mD(\Ssym) \cap \ker V$.
   Proposition~\ref{prop:resolventf} shows that $h\in H_{sa}(\Ssym) = \{0\}$.
   Therefore $h=0$ and $f_1, \dots, f_{\mu+1}$ are linearly dependent.
\end{proof}

The next theorem gives a criterion for the existence of a non-trivial selfadjoint subspace of $\widehat A$ if $\Ssym$ is cnsa.
Recall that a necessary condition is that $\dim W \ge 1$.

\begin{theorem}
   \label{thm:existenceHsa}
   Assume that $\Ssym h\in \ker V$ for all $h\in \mD(\Ssym)\cap \ker V$
   and that $\Ssym$ is cnsa.
   Let $W$ as in Lemma~\ref{lem:W}.
   Then $H_{sa}(\widehat A) = \gen\{ \phi : \phi \text{ satisfies \ref{item:a},   \ref{item:b}, \ref{item:c} below} \}$ where
   \begin{enumerate}[label={\upshape(\alph*)}]
      \item\label{item:a}  $\phi\in\mD(\widehat A)\setminus \mD(\Ssym )$,
      \item\label{item:b}  $\widehat A \phi \in \gen\{\phi\}$,
      \item\label{item:c}  $\Big(\im( \wC^* \wB^{*-1}) - \frac{1}{4} M_{\mathtt K} \Big) \wGamma_1 \phi = 0$ and $V\phi = \frac{1}{2\I} (\wGamma_1 \phi)^t \mathtt K$.
   \end{enumerate}

\end{theorem}

\begin{proof}
  
   Recall that $H_{sa}(\widehat A)$ is at most $\ell$-dimensional, 
   so $H_{sa}(\widehat A)= \gen\{ \psi_1, \dots, \psi_p \} \subseteq \Hsym(\widehat A)$ 
   where $\psi_1, \dots, \psi_p$ is a basis of eigenvectors of $\widehat A|_{H_{sa}(\widehat A)}$, where $p\leq \ell$.
   Clearly, each $\psi_j$ must satisfy \ref{item:a}, \ref{item:b} and \ref{item:c}, hence
   $H_{sa}(\widehat A) \subseteq \gen\{ \phi : \phi \text{ satisfies \ref{item:a},   \ref{item:b}, \ref{item:c}} \}$.
   
   On the other hand, if $\phi\in\mD(\widehat{A})$ satisfies \ref{item:c}, by \eqref{eq:ASymSpaceBT}, it belongs to $\Hsym(\widehat A)$ and consequently $\widehat{A}\phi=\widehat{A}^*\phi$. Moreover, by \ref{item:b}, it is an eigenvector of $\widehat{A}$. Let $z\in\mathbb{C}$ be the corresponding eigenvalue. We then get
   \begin{equation*}
    z\langle \phi,\phi\rangle=\langle \phi,\widehat{A}\phi\rangle=\langle \phi,\widehat{A}^*\phi\rangle=\langle \widehat{A}\phi,\phi\rangle=\overline{z}\langle \phi,\phi\rangle\:,
   \end{equation*}
   which implies $z\in\mathbb{R}$
   and thus $\phi\in H_{sa}(\widehat A)$.
   
\end{proof}


\section{Complete non-selfadjointness of the Laplacian on the interval $[0,1]$} 
\label{sec:cnsa:Laplace}

Recall that $S$ is the minimal Laplacian on the interval $(0,1)$ as defined in \eqref{eq:MinS}.
Let $V$ be a bounded non-negative operator on $L_2(0,1)$ and set $A=S+\I V$.
Then all maximally dissipative extensions can be written as in Corollary~\ref{cor:MaxDisRestA}.
Note that the domain of such an extension coincides with exactly one domain of a maximally dissipative extension of $S$.
All possible such domains are listed in Theorem~\ref{thm:MaxDisRestS}.
\smallskip

\begin{lemma}
   \label{lem:onlyEigValues}
   Let $V\ge 0$ be a bounded operator.
   \begin{enumerate}[label={\upshape(\roman*)}]

      \item Any maximally dissipative extension of the minimal Laplacian $S$ has only discrete spectrum.

      \item Any maximally dissipative extension of $A = S + \I V$ has only discrete spectrum.
      
   \end{enumerate}
\end{lemma}
\begin{proof}
   It suffices to show (ii) since (i) is the special case $V=0$.
   Note that the Dirichlet Laplacian $S_D := \widehat S_{\id, 0}$ is a selfadjoint extension of $S$ which has only point spectrum, therefore its resolvent is compact.
   By a standard argument, the resolvent of $S_D+\I V$ is also compact:
   \begin{equation*}
     (S_D+\I V-\lambda)^{-1}
     =[(S_D-\lambda)(I+(S_D-\lambda)^{-1}\I V)]^{-1}
     =(I+(S_D-\lambda)^{-1}\I V)^{-1}(S_D-\lambda)^{-1}\:,
   \end{equation*}
   where $\lambda\in\C^{-}$ is chosen such that $\|(S_D-\lambda)^{-1}V\|<1$ and thus $(I+(S_D-\lambda)^{-1}\I V)$ is boundedly invertible. 
   Now, let $\widehat{A}$ be any maximally dissipative extension of $S+\I V$. 
   Since for $\lambda\in\C^-$ the resolvents $(\widehat{A}-\lambda)^{-1}$ and $(S_D+\I V-\lambda)^{-1}$ coincide on $\range(S+\I V-\lambda)$ and $\dim \range(S+\I V-\lambda)^\perp=\dim\ker(S^*-\I V-\overline{\lambda})=2$, their difference $(\widehat{A}-\lambda)^{-1}-(S_D+\I V-\lambda)^{-1}$ is at most rank-two. 
   Consequently, $(\widehat{A}-\lambda)^{-1}$ is also compact and thus any maximally dissipative extension $\widehat{A}$ of $S+\I V$ has discrete spectrum.
\end{proof}
The lemma above together with \eqref{prop:KMNRealPointSpec} leads to the following proposition about the selfadjoint subspaces of the maximally dissipative realizations of the Laplacian with a complex potential.

\begin{proposition}
   \label{prop:CritForCNSA}
   Let $V\ge 0$ be a bounded operator.
   A maximally dissipative extension $\widehat A$ of $A = S + \I V$ is cnsa if and only if it has no real spectrum.
   Its (possibly trivial) maximal selfadjoint subspace is
   \begin{equation*}
      H_{sa}(\widehat A) = \overline{\gen}\{ \ker(\widehat A - \lambda) : \lambda\in\R\}.
   \end{equation*}
\end{proposition}

The following facts are well-known.
\begin{lemma}
   \label{lem:EVinEveryExt}
   If $S$ is a closed symmetric operator and $\lambda\in\R$ an eigenvalue for every selfadjoint extension $\widehat S$, then $\lambda$ is an eigenvalue of $S$.
   In particular $\ker(S-\lambda)\subseteq H_{sa}(S)$ and $S$ is not cnsa. 
\end{lemma}

\begin{lemma}
   \label{lem:doubleEV}
   If $S$ is a closed symmetric operator with defect indices $(1,1)$ and there exists $\lambda\in\R$ and a selfadjoint extension $\widehat S$ of $S$ such that such that $\lambda$ is an eigenvalue of $\widehat S$ of multiplicity larger or equal $2$, then $\lambda$ is an eigenvalue of $S$.
   In particular $\ker(S-\lambda)\subseteq H_{sa}(S)$ and $S$ is not cnsa. 
\end{lemma}

We recall that all the eigenvalues of the Laplacian with periodic or antiperiodic boundary conditions have multiplicity $2$.
More precisely:
\begin{lemma}
   \label{lem:antiperiodic}
   Let $S_{p}$ and $S_{ap}$ be defined by $S_{p/ap} f = S^*f = -f''$ with domains
   \begin{align*}
      \mD(S_{p}) &= \left\{ f\in H^2(0,1) : f(0) = f(1),\, f'(0) = f'(1) \right\}
      \\
      \mD(S_{ap}) &= \left\{ f\in H^2(0,1) : f(0) = -f(1),\, f'(0) = -f'(1) \right\}.
   \end{align*}
   Then
   \begin{itemize}

      \item
      $\sigma(S_{p}) = \{ (2n\pi)^2 : n\in \N_0 \}$. 
      The eigenvalue $0$ is simple and has eigenfunction $f_0(x) = 1$.
      All other eigenvalues have multiplicity $2$ with eigenfunctions $f_n(x) = \cos(2n\pi x)$ and $g_n(x) = \sin(2n\pi x)$, $n\ge 1$.

      \item
      $\sigma(S_{ap}) = \{ ((2n+1)\pi)^2 : n\in \N_0 \}$. 
      All eigenvalues have multiplicity $2$ with eigenfunctions $f_n(x) = \cos( (2n+1)\pi x)$ and $g_n(x) = \sin((2n+1)\pi x)$.
   \end{itemize}

   In particular, no one-dimensional symmetric restriction of $S_{p}$ or $S_{ap}$ can be cnsa.

\end{lemma}

Note that the domains of $S_{p}$ and $S_{ap}$ correspond to case \ref{item:caseIIS} in Theorem~\ref{thm:MaxDisRestS} with $c_{12} = 0$ and $c_{22} = 1$ (for the periodic case) and $c_{22} = -1$ (for the antiperiodic case).
\medskip

Recall the boundary triple $(\Gamma_0, \Gamma_1, \C^2)$ from \eqref{eq:BoundaryTripleSstar}.
Then, by Proposition~\ref{prop:KurasovCriterion}, every maximally dissipative extension of the minimal Laplacian $S$ is of the form 

\begin{equation*}
   \mD(S_{B,C}) = \{ f\in H^2(0,1) : B\Gamma_0 f = C\Gamma_1 f \}
   = \left\{ f\in H^2(0,1) : 
   (B | -C) \begin{pmatrix} \Gamma_0 f \\ \Gamma_1 f
   \end{pmatrix}
   = \begin{pmatrix} 0  \\ 0
   \end{pmatrix}
   \right\}
\end{equation*}
with $2\times 2$ matrices $B, C$ such that $\rk(B|-C) = 2$ and $\im(BC^*)\ge 0$.
Its adjoint is, see Lemma~\ref{lem:AdjointBT},
\begin{align}
   \label{OLD:eq:AdjointBC}
   \mD(S_{B,C}^*) 
   &= 
   \left\{ g\in H^2(0,1) : 
   \Gamma_0 g = C^* \begin{pmatrix} a \\ b
   \end{pmatrix},\
   \Gamma_1 g = B^* \begin{pmatrix} a \\ b
   \end{pmatrix}
   \ \text{ for }\
   \begin{pmatrix} a \\ b
   \end{pmatrix} \in\C^2
   \right\}
\end{align}
and its symmetric subspace is, see Theorem~\ref{thm:SymSpaceBT},
\begin{align}
   \label{eq:SymSpaceBC}
    \Hsym(S_{B,C)}
    = \left\{ f\in H^2(0,1) : 
    \Gamma_0 g = C^* \begin{pmatrix} a \\ b
    \end{pmatrix},\
    \Gamma_1 g = B^* \begin{pmatrix} a \\ b
    \end{pmatrix}
    \ \text{ for }\
    \begin{pmatrix} a \\ b
    \end{pmatrix} \in\ker( \im(BC^*) )
    \right\}.
\end{align}
We set
\begin{equation}
   \Ssym[, B,C]:
   \qquad \mD(\Ssym[,B,C]) = \Hsym(S_{B,C}),
   \quad 
   \Ssym[,B,C]f = S^*f.
\end{equation}
\bigskip

A concrete condition for the hypothesis $\Ssym(\mD(\Ssym)\cap\ker V)\subseteq \ker V$ in Proposition~\ref{prop:resolventf} to hold is the following.

 \begin{lemma}[\protect{\cite[Lemma~3.11]{FNWcnsa2024}}]
   \label{lem:FNWLemma311}
   Let $V$ be a multiplication operator by a non-trivial non-negative function $V(\cdot)\in L_\infty(0,1)$ and let $g\in H^n(0,1)\cap\ker V$.
   Then $g^{(k)}\in\ker V$ for every $0 \le k \le n$.
   
\end{lemma}

\begin{theorem}
   \label{thm:EV}
   Let $V$ be a multiplication operator by a non-trivial non-negative function $V(\cdot)\in L_\infty(0,1)$ and assume that 
   $\dim W=p$ where $W$ is the space from Lemma~\ref{lem:W}.
   If 
   there exists an open interval $(a,b)\subseteq [0,1]$
   such that $V>0$ a.e.\ in $(a,b)$,
   then $\widehat A$ has at most $p$ real eigenvalues counted with multiplicities.
\end{theorem}
\begin{proof}
   Since $\dim (\mD(\widehat A)/\mD(S)) \le 2$, the possible values for $p$ are $p=0, 1, 2$.
   
   By assumption on $V$, there exists an interval $(a,b)\subseteq (0,1)$ such that $V(x)>0$ for all $x\in (a,b)$.
   If $\phi$ is an eigenfunction of $\widehat A$ with a real eigenvalue $\lambda$, then $\phi\in H_{sa}(\widehat A)$ and therefore 
   $-\phi'' -\I V\phi= \widehat A^* \phi = \widehat A \phi = \lambda\phi$.
   Hence, by the uniqueness of solutions of the differential equation $-f''-\I V f = \lambda f$, we have the following implication for such $\phi$:
   \begin{equation}
       \label{eq:NullAufIntervall}
       \phi(x) = 0, \quad x\in (a,b)
       \qquad\implies\qquad
       \phi(x) = 0, \quad x\in (0,1).
   \end{equation}
   
   If $p=0$, then $\Hsym(\widehat A) = \mD(\Ssym)\cap \ker V$.
   If $\widehat A$ had a real eigenvalue $\lambda$ with eigenfunction $\phi$, then $\phi\in H_{sa}(\widehat A)\cap\mD(\widehat A) \subseteq \Hsym(\widehat A) \subseteq \ker V$, hence $\phi(x) = 0$ for all $x\in(a,b)$. 
   Therefore $\phi=0$ by \eqref{eq:NullAufIntervall}.

   The case $p=1$ is analogous to the case $p=2$, but a bit simpler, so it is omitted here.
   
   Now let us assume that $p=2$. 
   By contradiction, assume that $\widehat A$ has a set of linearly independent eigenfunctions $\phi_1,\, \phi_2,\,  \phi_{3}$ with real eigenvalues $\lambda_1,\, \lambda_2,\, \lambda_{3}$.
   Since $p=2$, there exist $\alpha_1,\alpha_2,\alpha_3$ not all equal to zero such that $\eta:=\alpha_1\phi_1+\alpha_2\phi_2+\alpha_3\phi_3\in\mD(\Ssym)\cap\ker(V)$. As above, it follows that 
   $-\phi_j''-\I V \phi_j = \lambda_j \phi_j$, $j=1,2,3$.
   Note that at most two eigenvalues can be equal because $S$ has no eigenvalues and $\widehat A$ is a two-dimensional extension of $S$.
   Let us first assume that $\lambda_1=\lambda_2\neq \lambda_3$.
   Then we set $\widetilde\phi = \alpha_1\phi_1 + \alpha_2\phi_2$.
   Note that $\widetilde\phi$ is an eigenfunction of $\widehat A$ with the real eigenvalue $\lambda_1$.
   Since $\eta\in \ker V$, we have $\eta(x)=\widetilde \phi(x) + \alpha_3\phi_3(x)  = 0$ for all $x\in (a,b)$ and moreover $-\eta'' = \Ssym\eta = \widehat A \eta = \lambda_1\widetilde \phi + \lambda_3 \alpha_3\phi_3$.
   Therefore, 
   \begin{equation}
       0 =  \lambda_1 \eta(x) - \eta''(x)
       =  (\lambda_1 - \lambda_3)\alpha_3\phi_3(x), 
       \qquad
       x\in (a,b).
   \end{equation}
   If $\alpha_3 = 0$, then $\widetilde \phi(x) = 0$ for all $x\in(a,b)$, hence in $(0,1)$ by \eqref{eq:NullAufIntervall}. 
   But then $\phi_1$ and $\phi_2$ are linearly dependent, in contradiction to our assumption.
   If on the other hand $\alpha_3\neq 0$, then 
   $\phi_3(x) = 0$ on $(a,b)$, hence on $(0,1)$ by \eqref{eq:NullAufIntervall}. 

   If the eigenvalues  $\lambda_1, \lambda_2, \lambda_3$ are pairwise distinct, then similar as above, we obtain that
   \begin{align*}
      \left.
       \begin{aligned}
          0 &= \eta(x) 
          = \alpha_1\phi_1(x) + \alpha_2\phi_2(x) + \alpha_3\phi_3(x), \\
          \quad\text{and}\quad
          0 &= -\eta''(x) = \widehat A^*\eta(x) 
          = \alpha_1\lambda_1 \phi_1(x) + \alpha_2 \lambda_2 \phi_2(x) + \alpha_3\lambda_3 \phi_3 (x),
          \\
          \quad\text{and}\quad
          0 &= \eta^{(4)}(x) = \left(\widehat A^*\right)^2\eta(x)
          = \alpha_1\lambda_1^2 \phi_1(x) + \alpha_2 \lambda_2^2 \phi_2(x) + \alpha_3\lambda_3^2 \phi_3 (x),
       \end{aligned}
       \right\}
       \qquad \text{for }\ x\in(a,b).
   \end{align*}
   Since the $\lambda_j$ are pairwise different and the Vandermonde matrix $(\lambda_j^k)_{k,j}$ is invertible, 
   it follows that $\alpha_j\phi_j(x) = 0$ for all $x\in(a,b)$ and $j=1,2,3$.
   Since not all $\alpha_j$ can vanish, there exists at least one $j_0$ such that $\phi_{j_0}(x) = 0$ on $(a,b)$, hence on $(0,1)$ contradicting the fact that the $\phi_j$ are eigenfunctions of $\widehat A$.
\end{proof}
\bigskip

In the next sections we will discuss in detail some of the dissipative extensions of $S$ and of $A=S+\I V$ which illustrate the different situations that can arise for complete non-selfadjointness depending on the chosen boundary conditions and the value of $\rk(\im BC^*)$.

\subsection{The selfadjoint case: $\rk(\im (BC^*) ) = 0$.} 
\label{subsec:5:1}
If $\rk (\im (BC^*)) = 0$, then the operator $\widehat S = S_{B, C}$ is selfadjoint. 
So  the only dissipative contribution in $\widehat A$ comes from the term $\I V$.
Since $V$ is bounded, we have $\mD(\widehat A^*) = \mD(\Ssym^*) = \mD(\Ssym) = \mD(\widehat A)$ and $\widehat A^* = \Ssym - \I V$.
Moreover, $\Hsym(\widehat A) = \mD(\widehat A)\cap \ker V = \mD(\Ssym)\cap \ker V$.

In this particular case, the selfajdoint subspace $H_{sa}(\widehat{S}+\I V)$ will depend on the particular choice of the operator $V$.

For instance, if $V$ is the orthogonal projection onto a spectral subspace of $\Ssym$, then 
$(1-V)\mD(\Ssym) \subseteq H_{sa}(\widehat A)$. 
Similarly, if $V=P\widetilde V P$ where $P$ is the orthogonal projection onto a spectral subspace of $\Ssym$ and $\widetilde V \ge 0$ is a bounded operator, we have that $(1-P)\mD(\Ssym)\subseteq H_{sa}(\widehat A)$.
\smallskip

However, if we assume $V$ to be a multiplication operator by an a.e. bounded and non-negative function $V(x)$, then we can conclude that $\widehat{S}+\I V$ is cnsa:\footnote{The proof of this result was kindly provided to us by Pavel Kurasov.}

\begin{theorem} \label{thm:sa+pot=cnsa}
Let $\widehat{S}$ be a selfadjoint extension of the minimal Laplacian $S$ and assume that $V$ is a multiplication operator by a function $V(\cdot)\in L_\infty(0,1)$ with $V\geq 0$ a.e.\ and $V>0$ a.e.\ on some Lebesgue measurable set $\Delta$ with Lebesgue measure $|\Delta| > 0$.
Then, $\widehat{A}=\widehat{S}+\I V$ is cnsa.
\end{theorem}
\begin{proof}
Suppose that there is $\eta\neq 0$ such that $\widehat{A}\eta=r\eta$ for some $r\in\R$. Then
    \begin{equation}
        \widehat{S}\eta+iV\eta=r\eta.
    \end{equation}
    Taking the inner product with $\eta$ on both sides of the previous equation, we obtain 
    \begin{equation}
        \langle \eta,\widehat{S}\eta\rangle+i\langle \eta,{V}\eta\rangle=r\|\eta\|^2.
    \end{equation}
    Now, taking the imaginary part we get \begin{equation}
        \im\langle \eta,\widehat{S}\eta\rangle+\langle \eta, V\eta\rangle=\im\langle \eta,\widehat{S}\eta\rangle+\|V^{1/2}\eta\|^2=0,
    \end{equation} 
    which implies $V\eta=0$ since $\widehat{S}$ is selfadjoint. Therefore, we have that
    \begin{equation}
        \widehat{S}\eta=r\eta,
    \end{equation}
    which means that $\eta$ is also an eigenfunction of the operator $\widehat{S}$ and thus has at most finitely many zeros $x_1,\ldots,x_n$. 
    Define
    \begin{equation}
        I:=\bigcup_{j=1}^nI_j
    \end{equation}
    where each $I_j$ denotes an interval around $x_j$ such that $(I_j)_{j=1}^n$ are pairwise disjoint and the Lebesgue measure $|I|=\sum_{j=1}^n|I_j|<|\Delta|/2$
    with the set $\Delta$ as in the statement of the theorem.
    Now, since $\eta$ is a continuous function and $\eta\neq 0$ on the compact set $[0,1]\setminus{I}$, there exists $d>0$ such that $|\eta|\geq d$ on $[0,1]\setminus{I}$. Consequently,
    \begin{align}
        0=\langle\eta,V\eta\rangle\geq\int_{\Delta\cap([0,1]\setminus{I})}{(V(x))}|\eta(x)|^2 \rd x\geq d^2\int_{\Delta\cap([0,1]\setminus{I})}V(x)\; \rd x>0
    \end{align}
    which is a contradiction. 
    Therefore, we proved that $\ker (\widehat{A} - r) = \{0\}$ for every $r\in\R$.
\end{proof}


\subsection{At most one-dimensional dissipative extensions: $\rk(\im (BC^*) ) \le 1$} 
\label{subsec:5:2}

\begingroup 
We will discuss in detail the case \ref{item:caseIIS} from Theorem~\ref{thm:MaxDisRestS}.
The cases \ref{item:caseIIIS} and \ref{item:caseIVS} can be analysed similarly.

Let $\widehat S$ be the operator from case~\ref{item:caseIIS} in Theorem~\ref{thm:MaxDisRestS}.
If we set $b:= c_{22}$ and $c:= c_{12}$, then 
$B= \begin{pmatrix}
   1 & -\overline b^{-1} \\ 0 & 0      
\end{pmatrix}
$
and 
$C= \begin{pmatrix}
   0 & c \\ 1 & b
\end{pmatrix}.
$
Recall that $b\neq 0$ and $\im(\overline{b} c) \ge 0$.
The symmetric part of $\widehat S$ is
\begin{align*}
   \Ssym:\qquad \mD(\Ssym) = \Hsym(\widehat S),\quad \Ssym f = S^* = -f''.
\end{align*}
Clearly,  
$\Ssym \subseteq \widehat S \subseteq \Ssym^*$ and 
$\Ssym \subseteq \widehat S^* \subseteq \Ssym^*$.

\begin{itemize}
    \item  
    If $\im \overline{b}c = 0$, then $\im(BC^*) = 0$, $\widehat S$ is selfadjoint, $\Hsym(\widehat S) = \mD(\widehat S)$ and $\Ssym = \widehat S$.

    \item 
    If $\im \overline{b}c > 0$, then it is easy to check that 
    \begin{align*}
       \mD(\Ssym) &= 
       \left\{ f\in H^2(0,1) : \
       \begin{aligned}
       f(1) - \overline{b} f(0) &= \overline{b} c f'(1) \\
       f(1) - \overline{b} f(0) &= b\overline{c} f'(1) \\
       f'(0) - b f'(1) &= 0
       \end{aligned}
       \right\}
       = 
       \left\{ f\in H^2(0,1) : \
       \begin{aligned}
       f(1) - \overline{b} f(0) &= 0 \\
       f'(0) = f'(1) &= 0
       \end{aligned}
       \right\}, 
       \\[1ex]
       \mD(\Ssym^*) &= 
       \left\{ f\in H^2(0,1) : \ f'(0) - b f'(1) = 0 \right\},
       \\[1ex]
       \mD(\widehat S) &= 
       \left\{ f\in H^2(0,1) : \
       \begin{aligned}
       f(1) - \overline{b} f(0) &= \overline{b} c f'(1) \\
       f'(0) - b f'(1) &= 0
       \end{aligned}
       \right\}
       =
       \left\{ f\in \mD(\Ssym^*) : \
       f(1) - \overline{b} f(0) = \overline{b} c f'(1)
       \right\}, 
       \\[1ex]
       \mD(\widehat S^*) &= 
       \left\{ f\in H^2(0,1) : \
       \begin{aligned}
       f(1) - \overline{b} f(0) &= b\overline{c} f'(1) \\
       f'(0) - b f'(1) &= 0
       \end{aligned}
       \right\}
       =
       \left\{ f\in \mD(\Ssym^*) : \
       f(1) - \overline{b} f(0) = b\overline{c} f'(1) 
       \right\},
    \end{align*}
    hence $\Hsym(\widehat S)$ is a one-dimensional extension of the domain $\mD(S)$ of the minimal operator.
\end{itemize}

\begin{lemma}
   Assume $\im \overline{b}c > 0$.
   Then $(\wGamma_0, \wGamma_1, \C)$ is a boundary triple for $\Ssym^*$ where 
   \begin{align}
      \label{eq:wGammaBC}
      \begin{aligned}
	 \wGamma_0 &:\mD(\Ssym^*)\to \C,\ \wGamma_0 f = \overline{b} f(0) - f(1), \\ 
	 \wGamma_1 &:\mD(\Ssym^*)\to \C,\ \wGamma_1 f = f'(1).
      \end{aligned}
   \end{align}
   In particular, all dissipative extensions of $\Ssym$ are given by 
   \begin{equation*}
      \Ssym[, \alpha]: \qquad
      \mD(\Ssym[, \alpha]) =
      \left\{ f\in \mD(S^*) : \wGamma_0 f = \alpha \wGamma_1 f\right\},
      \quad \Ssym[, \alpha] f = S^* f = -f''
   \end{equation*}
   for $\im \alpha \ge 0$.
   The extension $\Ssym[, \alpha]$ is selfadjoint if and only if $\alpha \in \R\cup \{\infty\}$, with the usual convention that $\alpha = \infty$ corresponds to the condition $\wGamma_1 f = 0$.

   The extension $\widehat S$ corresponds to $\alpha = -\overline{b} c$, while 
   the extension $\widehat S^*$ corresponds to $\alpha = -b\overline{c}$.
\end{lemma}
\begin{proof}
   Clearly, $(\wGamma_0,\, \wGamma_1): \mD(\Ssym^*)\to \C^2$ is surjective. 
   Moreover, for all $f,g\in \mD(\Ssym)$ we have 
   \begin{align*}
      \scalar{\Ssym f}{g} - \scalar{f}{\Ssym^* g}
      &= -\overline{f'(1)}\, g(1) + \overline{f'(0)}\, g(0) + \overline{f(1)}\, g'(1) - \overline{f(0)}\, g'(0)
      \\
      &= -\overline{f'(1)}\, g(1) + \overline{b f'(1)}\, g(0) + \overline{ f(1) }\, g'(1) - \overline{f(0)}\,  bg'(1)
      \\
      &= 
      \overline{f'(1)}\,  \big[ \overline{b} g(0) - g(1) \big]
      -
      \big[ \overline{ \overline{b} f(0) - f(1) }\big] g'(1)
      = \overline{\wGamma_1 f}\, \wGamma_0 g - \overline{\wGamma_0 f}\, \wGamma_1 g.
      \qedhere
   \end{align*}
\end{proof}
\bigskip

In conclusion, we can write the domain of $\widehat S$ as
\begin{align}
   \label{eq:mDwidehatS}
   \mD(\widehat S)
   &= \{ f\in \mD(\Ssym^*) : \wGamma_0 f = -\overline b c\, \wGamma_1 f \}
   = \mD(\Ssym) \dot + \mU
\end{align}
with 
$\mU = \{0\}$ if $\im(c\overline b) = 0$,
and
$\mU = \gen\{ u \} \subseteq \mD(\Ssym^*)//\mD(\Ssym)$ with $\dim\mU = 1$, where we choose $u$ such that $(u(0), u(1), u'(0), u'(1)) = (0, c\overline b, b, 1)$.
\smallskip

The question of complete non-selfadjointness of $\Ssym$ is settled in the next lemma.
\begin{lemma}
   \label{lem:widetildeScnsa}
   \begin{enumerate}[label={\upshape(\roman*)}]

      \item 
      If $\im \overline b c =0$, then $\Ssym$ is selfadjoint, hence $H_{sa}(\Ssym) = L_2(0,1)$.

      \item 
      If $\im \overline b c > 0$ and $b = \pm 1$, then $\Ssym$ is not selfadjoint and not cnsa.

      \item 
      If $\im \overline b c > 0$ and $b \neq \pm 1$, then $\Ssym$ is cnsa.

   \end{enumerate}
\end{lemma}
\begin{proof}
   The first claim is clear. 
   For the second claim, we observe that either the periodic or the antiperiodic Laplacian from Lemma~\ref{lem:antiperiodic} is a selfadjoint extension of $\Ssym$, hence it is not cnsa.
   More precisely, $\overline{\gen}\{ \cos(2n\pi \cdot) : n\in\N_0\} \subseteq H_{sa}(\Ssym)$ if $b=1$ and 
   $\overline{\gen}\{ \sin(2n\pi \cdot) : n\in\N\} \subseteq H_{sa}(\Ssym)$ if $b=-1$.
   For the third claim, if $\im \overline b c >0$ and $b\neq \pm 1$, then it is easy to verify by direct calculation that $\Ssym$ has no eigenvalues, hence it is cnsa by Remark~\ref{prop:BHdS348}.
\end{proof}
\bigskip

Now let us return to the study of the dissipative extensions $\widehat A$ of $S+\I V$ which have domain $\mD(\widehat A) = \mD(\widehat S)$.
This means that they are extensions of $\Ssym + \I V$ where we recall that $\Ssym$ is a symmetric operator.
We already have a full description of all such extensions by Corollary~\ref{cor:MaxDisRestA}\ref{item:caseIIA}.
However here we want to use the boundary triple \eqref{eq:wGammaBC}.
\smallskip

Recall that the space $\mU$ in \eqref{eq:mDwidehatS} has $\dim\mU = 0$ if and only if $\im(\overline{b}c) = 0$.
Otherwise
\begin{align*}
   \mD(\widehat A) 
   = \left\{ f\in \mD(\Ssym^*) : 
   \wGamma_0 g =  -\overline{b} c\, \wGamma_1 g
   \right\}
   = \mD(\Ssym) \dot + \gen\{u\}
\end{align*}
with $u\in\mD(\Ssym^*)$ such that $(u(0), u(1), u'(0), u'(1)) = (0, c\overline{b}, b, 1)$.
In this case, $\Ssym$ is symmetric but not selfadjoint.

Every $f\in \mD(\widehat S)$ can be written as 
$f = f_0 + f'(1)u = f_0 + (\wGamma_1 f)u$.
We also know that $Af = (S^* + \I V)f + f'(1)k$ for some $k\in\range V$ because $\dim( \mD(\widehat A) / \mD(\Ssym) ) = 1$.

By \cite[Theorem 6.3]{nonproper2018}, a necessary and sufficient condition on $k$ for $\widehat A$ to be maximally dissipative is that $\im \scalar{\Ssym^* u}{u} \ge \frac{1}{4}  \| V^{-1/2} k \|^2$.
Using that
\begin{align*}
   \im \scalar{\Ssym^* u}{u}  
   = \frac{1}{2\I } \big( \scalar{\Ssym^* u}{u} - \scalar{u}{\Ssym^* u} \big)
   = \frac{1}{2\I } \big( \overline{\wGamma_1 u}\, \wGamma_0 u - \overline{\wGamma_0 u}\, \wGamma_1 u \big)
   = \im( \overline{b} c )\, |\wGamma_1 u |^2 
   = \im( \overline{b} c )\:,
\end{align*}
the condition on $k$ is 
\begin{equation*}
   \im \overline{b} c \ge \frac{1}{4}  \| V^{-1/2} k \|^2.
\end{equation*}
According to Theorem~\ref{thm:widehatAadjoint}
the adjoint operator of $\widehat A$ and its symmetric subspace are 
\begin{align*}
   \widehat A^*:\qquad
   \widehat A^*f = (S^* - \I V)f,
   \qquad
   \begin{aligned}[t]
      \mD(\widehat A^*) 
      &= \left\{ f\in \mD(\Ssym^*) : 
      \wGamma_0 f = b\overline{c}\, \wGamma_1 f - \scalar{k}{f}
      \right\}
    \end{aligned}
\end{align*}
and
\begin{align*}
   \Hsym(\widehat A) 
   &= \left\{ f\in \mD(\widehat A) : 
   \Big( \im(\overline{b} c)  - \frac{1}{4} \|V^{-1/2}k \|^2 \Big) \wGamma_1 f = 0,\
   2\I V f = (\wGamma_1 f) k
   \right\}.
\end{align*}

In order to study selfadjoint subspaces of $\widehat A$
we have to distinguish between the \define{non-critical case} when $\im \overline{b}c > \frac{1}{4} \| V^{-1/2}k\|^2$
and the \define{critical case} when  $\im \overline{b}c = \frac{1}{4} \| V^{-1/2}k\|^2$.
\smallskip

In the non-critical case, the dissipativity arising purely from the boundary conditions (i.e., from the subspace $\mU$), is stronger than the effect resulting from the non-local action of $\widehat A$ on $\mU$.
On the other hand, in the critical case, the dissipativity from the boundary condition balances exactly the effect resulting from the non-local action of $\widehat A$ on $\mU$.

\begin{lemma}
   \label{lem:WBC}

   \begin{enumerate}[label={\upshape(\roman*)}]

      \item
      If $\im \overline{b}c > \frac{1}{4} \| V^{-1/2}k\|^2$, then $\Hsym(\widehat A)\subseteq \ker V$, in particular
      \begin{equation*}
	 \Hsym(\widehat A) = \mD(\Ssym) \cap \ker V
      \end{equation*}
      and $\widehat A f = \Ssym f$ for every $f\in \Hsym(\widehat A)$.

      \item
      If $\im \overline{b}c = \frac{1}{4} \| V^{-1/2}k\|^2$, then 
      \begin{equation*}
	 \Hsym(\widehat A) = ( \mD(\Ssym) \cap \ker V ) \dot + W
      \end{equation*}
      where $W$ has dimension at most $1$ and $W\cap \mD(\Ssym) = W \cap \ker V = \{0\}$.

   \end{enumerate}

\end{lemma}
\begin{proof}
This is  Lemma~\ref{lem:W}.
\end{proof}

\begin{proposition}
   \label{prop:keineAhnungBC}
   \begin{enumerate}[label={\upshape(\roman*)}]

      \item
      If $G\subseteq \Hsym(\widehat A)\cap \ker V$ is a selfadjoint subspace for $\widehat A$, then $G\subseteq H_{sa}(\Ssym)$.

      \item
      If $\im \overline{b}c > \frac{1}{4} \| V^{-1/2}k\|^2$, then 
      $H_{sa}(\widehat A)\subseteq H_{sa}(\Ssym)$.

   \end{enumerate}
\end{proposition}
\begin{proof}
This is Proposition~\ref{prop:keineAhnung}.
\end{proof}

\begin{proposition}
   \label{prop:resolventfBC}
   Assume that there exists $f\in H_{sa}(\widehat A)$ such that 
   $(\widehat A - \lambda)^{-1}f\in \mD(\Ssym)\cap \ker V$ 
   for some $\lambda\in\rho(\widehat A)$
   and that 
   \begin{equation}
   \label{eq:resolventfBC}
      \Ssym h\in \ker V
      \quad\text{for all }\quad
      h\in \mD(\Ssym)\cap \ker V.
   \end{equation}
   Then $f\in H_{sa}(\Ssym)$.
\end{proposition}
\begin{proof}
This is Proposition~\ref{prop:resolventf}.
\end{proof}
Note that the condition \eqref{eq:resolventfBC} is satisfied if for instance $V$ is a multiplication operator, see Lemma~\ref{lem:FNWLemma311}.
\smallskip

If in addition to the hypotheses in Proposition~\ref{prop:resolventfBC} we assume that $\Ssym$ is cnsa, then we can conclude that the selfadjoint subspace of $\widehat A$ is at most one-dimensional.
Recall that criteria for $\Ssym$ being cnsa are given in Lemma~\ref{lem:widetildeScnsa}.

\begin{theorem}
   \label{thm:dimleq1BC}
   Assume that $\Ssym h\in \ker V$ for all $h\in \mD(\Ssym)\cap \ker V$.
   In addition to the hypotheses of Proposition~\ref{prop:resolventfBC}, assume that $\Ssym$ is cnsa. 
   Let $W$ as in Lemma~\ref{lem:W}.
   \begin{enumerate}[label={\upshape(\roman*)}]

      \item If $W = \{0 \}$, then $H_{sa}(\widehat A) = \{0\}$.
      \item If $\dim W = 1$, then $\dim( H_{sa}(\widehat A) ) \le 1$.

   \end{enumerate}
\end{theorem}
\begin{proof}
This is Theorem~\ref{thm:dimleq1}.
\end{proof}

The next theorem sums up the criteria for the existence of a selfadjoint subspace of $\widehat A$ if $\Ssym$ is cnsa under the hypotheses of Proposition~\ref{prop:resolventf}.
Recall that $H_{sa}(\widehat A)$ is at most one-dimensional.

\begin{theorem}
   \label{thm:existenceHsaNew}
   Assume that $\Ssym$ is cnsa and that 
   $\Ssym h\in \ker V$ for all $h\in \mD(\Ssym)\cap \ker V$.
   The latter is true if $V$ is a multiplication operator.
   Let $W$ as in Lemma~\ref{lem:W}.
   \begin{enumerate}[label={\upshape(\roman*)}]

      \item 
      If $W = \{0\}$, then $H_{sa}(\widehat A) = \{0\}$.

      \item 
      If $\dim W = 1$, then $H_{sa}(\widehat A) = \gen\{ \phi \}$ where
      \begin{enumerate}[label={\upshape(\alph*)}]
	 \item\label{item:a:CaseII}  $\phi\in\mD(\widehat A)\setminus \mD(\Ssym )$,
	 \item\label{item:b:CaseII}  $\widehat A \phi \in \gen\{\phi\}$,
	 \item\label{item:c:CaseII}  $V\phi = \frac{1}{2\I} (\wGamma_1 \phi) k$
      \end{enumerate}
      if such a $\phi$ exists.
      Otherwise $H_{sa}(\widehat A) = \{ 0 \}$.
   \end{enumerate}

\end{theorem}

We want to give some more concrete conditions for the hypotheses of Proposition~\ref{prop:resolventf} to hold, which then lead to (non-)existence of a selfadjoint subspace of $\widehat A$.

\begin{theorem}
   \label{thm:5:15}
   Let $V$ be a multiplication operator with a non-trivial non-negative function $V(\cdot)\in L_\infty(0,1)$.
   If at least one of the conditions 
   \begin{itemize}
      \item $b\neq \pm 1$,
      \item there exists an open interval $(a,b)\subseteq [0,1]$ such that $V>0$ a.e.\ in $(a,b)$
   \end{itemize}
   is satisfied, then $\dim (H_{sa}(\widehat A)) = 1$ if and only if there exists $\phi$ satisfying 
   \ref{item:a}, \ref{item:b}, \ref{item:c}. 
   In this case $H_{sa}(\widehat A) = \gen\{\phi\}$.
   Otherwise $\dim (H_{sa}(\widehat A)) = 0$.
\end{theorem}
It should be noted that under the condition (ii),  $\Ssym$ is \emph{not} assumed to be cnsa. 
In fact, if $b=\pm 1$, then $\dim(H_{sa}(\Ssym)) = \infty$.
\begin{proof}
   The condition $b\neq \pm 1$ implies that $\Ssym$ is cnsa, see Lemma~\ref{lem:widetildeScnsa}.
   Hence the claim follows from Theorem~\ref{thm:existenceHsaNew}.
   If 
   $V>0$ a.e.\ in an open interval $(a,b)\subseteq [0,1]$,
   then $\widehat A$ has at most one real simple eigenvalue by Theorem~\ref{thm:EV}.
   Therefore, the claim follows from \eqref{prop:KMNRealPointSpec}.
\end{proof}
\endgroup

\subsection{Two-dimensional dissipative extensions} 
\label{subsec:5:3}

Assume that $\rk (\im(BC^*) ) = 2$.
This is possible only if both $B$ and $C$ are invertible. 
The invertibility of both $B$ and $C$ implies that $\Gamma_0$ and  $\Gamma_1$ restricted to $\mD(\widehat S)$ are surjective.
This follows from the fact 
$(B|-C) = B(\id|-B^{-1}C) = C(C^{-1}B|-\id)$.
\smallskip


From \eqref{OLD:eq:AdjointBC} and \eqref{eq:SymSpaceBC} we obtain 
\begin{align*}
   \mD(\widehat S) 
   &=  \left\{ f\in H^2(0,1) : \
   B\Gamma_0 f = C\Gamma_1 f
   \right\},
   \\[1ex]
   \mD(\widehat S^*) 
   &=  \left\{ f\in H^2(0,1) : \
   C^*{}^{-1}\Gamma_0 f = B^*{}^{-1} \Gamma_1 f
   \right\}.
   \\
   \Hsym(\widehat S) 
   &= \mD(\widehat S) \cap \mD(\widehat S^*)
   = \mD(S).
\end{align*}
Since $\Hsym(\widehat S) = \mD(S)$ there is no symmetric operator $\Ssym$ with $S\subsetneq \widehat S$.
So here $\Ssym = S$.
\smallskip

Now let us return to the study of the dissipative extensions $\widehat A$ of $S$ which have domain $\mD(\widehat A) = \mD(\widehat S)$ as in Corollary~\ref{cor:MaxDisRestA}.
This means that they are extensions of $S + \I V$ with the minimal symmetric operator $S$.
\smallskip

Observe that 
\begin{equation*}
   \mD(\widehat S) = \mD(\Ssym) \dot + \gen\{ u, v\}
\end{equation*}
where $u, v\in \mD(\widehat S)$ are linearly independent over $\mD(\Ssym)$.
We may choose for instance $u, v$ such that 
\begin{equation}
   \Gamma_1 u = \begin{pmatrix} 1 \\ 0
   \end{pmatrix},
   \qquad
   \Gamma_1 v = \begin{pmatrix} 0 \\ 1
   \end{pmatrix}.
\end{equation}

Hence every $f\in \mD(\widehat S)$ can be written as 
$f = f_0 + f'(0)u  - f'(1) v$ with $f_0\in\mD(\Ssym)$.
We also know that $Af = (S^* + \I V)f + f(0)k_1 - f(1) k_2$ for some $k_1, k_2\in\range V^{1/2}$ because $\dim( \mD(\widehat A) / \mD(\Ssym) ) = 2$.
This corresponds to the map, see Theorem~\ref{thm:ChristophMaxDiss},
\begin{equation*}
   \mL: \mU \to\range (V^{1/2}),\qquad
   \mL(\alpha u + \beta v) = \alpha k_1 + \beta k_2.
\end{equation*}

From Corollary~\ref{cor:MaxDissCritABT} we obtain that $\widehat A$ is maximally dissipative if and only if 
\begin{equation}
   \label{eq:DissCondV}
   \im( C^*B^*{}^{-1} )
   \ge 
   \frac{1}{4} 
   M_K
\end{equation}
where $M_{\texttt K}$ as in \eqref{eq:defMK}.
\medskip

According to Theorem~\ref{thm:widehatAadjoint}
the adjoint operator of $\widehat A$ and its symmetric subspace are 
\begin{align*}
   \widehat A^*:\qquad
   \widehat A^*f = (S^* - \I V)f,
   \qquad
   \begin{aligned}[t]
      \mD(\widehat A^*) 
      &= 
      \left\{ f\in \mD(\Ssym^*) : 
      \Gamma_0 f = C^*B^*{}^{-1}\Gamma_1 f -
      \begin{pmatrix} \scalar{k_1}{f} \\ \scalar{k_2}{f}
      \end{pmatrix}
      = 0
      \right\}
  \end{aligned}
\end{align*}
and
\begin{align}
   \label{eq:ASymSpaceBC}
   \Hsym(\widehat A)
   &= \left\{ f\in \mD(\widehat A) : 
   \Big( \im( C^* B^*{}^{-1})  - \frac{1}{4} M_K\Big) \Gamma_1 f = 0,\
   2\I V f = (\Gamma_1 f)^t 
   \begin{pmatrix} k_1 \\ k_2
   \end{pmatrix}
   \right\}.
\end{align}

Recall from Lemma~\ref{lem:W} that the dimension of $\ker( \im(C^*B^{*-1})-\frac{1}{4}M_{\mathtt K})$ is an upper bound for the dimension of $W$ where
\begin{equation*}
   \Hsym(\widehat A) = ( \mD(\Ssym) \cap \ker V ) \dot + W
   \qquad\text{and}\qquad
   W\cap \mD(\Ssym) = W \cap \ker V = \{0\}.
\end{equation*}
Note that this is analogous to Lemma~\ref{lem:WBC} where $\overline{b}c$ plays the role of $C^*B^{*-1}$ and $\frac{1}{4}\|V^{-1/2}k\|^2$  plays the role of $M_{\mathtt K}$.
In Section~\ref{subsec:5:2} we only have 
the critical case $\im (\overline{b}c) = \frac{1}{4}\|V^{-1/2}k\|^2$
and the non-critical case $\im(\overline{b}c > \frac{1}{4}\|V^{-1/2}k\|^2$,
while in our higher-dimensional setting here we also have the intermediate case  $\rk( \im (C^*B^*{}^{-1}) - \frac{1}{4} M_K ) = 1$.

\begin{lemma}
   \label{lem:WCaseI}
   \begin{enumerate}[label={\upshape(\roman*)}]
      \item
      If $\im(C^*B^{*-1}) > \frac{1}{4} M_K$, or equivalently, if
      $\dim(\ker( \im (C^*B^{*-1}) - \frac{1}{4} M_K )) = 0$, then 
      $\Hsym(\widehat A)\subseteq \ker V$.
      In particular
	 $\Hsym(\widehat A) = \mD(\Ssym) \cap \ker V$
      and $\widehat A f = \Ssym f$ for every $f\in \Hsym(\widehat A)$.

      \item
      If $\rk( \im( (C^*B^{*-1}) - \frac{1}{4} M_K )) = 1$, or equivalently, if
      $\dim(\ker( \im (C^*B^{*-1}) - \frac{1}{4} M_K )) = 1$, then 
      $\dim W \le 1$.

      \item
      If $\im (C^*B^{*-1}) = \frac{1}{4} M_K $, 
      or equivalently, if
      $\dim(\ker( \im (C^*B^{*-1}) - \frac{1}{4} M_K )) = 2$, then 
      $\dim W \le 2$.
   \end{enumerate}
\end{lemma}


In analogy to Proposition~\ref{prop:keineAhnungBC} we obtain the following.
\begin{proposition}
   \label{prop:keineAhnungCaseI}
   \begin{enumerate}[label={\upshape(\roman*)}]

      \item
      If $G\subseteq \Hsym(\widehat A)\cap \ker V$ is a selfadjoint subspace for $\widehat A$, then $G\subseteq H_{sa}(\Ssym)$.

      \item
      If $\im (C^* B^{*-1}) > \frac{1}{4} M_K$, then 
      $H_{sa}(\widehat A)\subseteq H_{sa}(\Ssym)$.

   \end{enumerate}
\end{proposition}
\begin{proof}
   \begin{enumerate}[label={\upshape(\roman*)}]
      \item
      The proof is the same as in Proposition~\ref{prop:keineAhnung}.
      
      \item
      If $\im C^* B^*{}^{-1} > \frac{1}{4} M_K$, then $\Hsym(\widehat A) \cap \ker V = \Hsym(\widehat A)$.
      Since every selfadjoint subspace of $\widehat A$ is contained in $\Hsym(\widehat A)$, it follows that $H_{sa}(\widehat A) \subseteq H_{sa}(\Ssym)$.
      \qedhere
   \end{enumerate}

\end{proof}

\begin{proposition}
   \label{prop:resolventfCaseI}
   Assume that there exists $f\in H_{sa}(\widehat A)$ such that 
   $(\widehat A - \lambda)^{-1}f\in \mD(\Ssym)\cap \ker V$ 
   for some $\lambda\in\rho(\widehat A)$
   and that 
   $\Ssym h\in \ker V$ for all $h\in \mD(\Ssym)\cap \ker V$.
   Then $f\in H_{sa}(\Ssym)$.
\end{proposition}
\begin{proof}
   This is Proposition~\ref{prop:resolventf}.
\end{proof}

Since our base operator $\Ssym = S$ is cnsa, we can conclude that the selfadjoint subspace of $\widehat A$ is at most two-dimensional if $\ker V$ is invariant under $S$.
Recall that this is the case if for instance $V$ is a multiplication operator, see Lemma~\ref{lem:FNWLemma311}.

\begin{theorem}
   \label{OLD:thm:dimleq1}
   Assume that $\Ssym h\in \ker V$ for all $h\in \mD(\Ssym)\cap \ker V$ and let $W$ as in Lemma~\ref{lem:WCaseI}.
   Then $\dim( H_{sa}(\widehat A) ) \le \dim W $.
\end{theorem}
\begin{proof}
   This is Theorem~\ref{thm:dimleq1}.
\end{proof}

The next theorem gives a criterion for the existence of a non-trivial selfadjoint subspace of $\widehat A$ in the critical case.
Recall that a necessary condition is that $\dim W \ge 1$.

\begin{theorem}
   \label{thm:existenceHsaCaseI}
   Assume that $\Ssym h\in \ker V$ for all $h\in \mD(\Ssym)\cap \ker V$ and that $\im(C^* B^{*-1}) = \frac{1}{4} M_{\mathtt K}$.
   Let $W$ as in Lemma~\ref{lem:W}.
   If $\dim W \ge 1$, then $H_{sa}(\widehat A) = \gen\{ \phi : \phi \text{ satisfies  \ref{item:a:CaseI}, \ref{item:b:CaseI}, \ref{item:c:CaseI}}\ \text{below} \}$ where
   \begin{enumerate}[label={\upshape(\alph*)}]
      \item\label{item:a:CaseI}
      $\phi\in\mD(\widehat A)\setminus \mD(\Ssym )$,
      
      \item\label{item:b:CaseI}
      $\widehat A \phi \in \gen\{\phi\}$,
      
      \item\label{item:c:CaseI}
      $V\phi = \frac{1}{2\I} (\Gamma_1 \phi)^t \mathtt K$
   \end{enumerate}
   and $ \dim(H_{sa}(\widehat A))\le 2$.

\end{theorem}
\begin{proof}
   We first recall that $H_{sa}(\widehat A)$ is at most two-dimensional, 
   $H_{sa}(\widehat A)= \gen\{ \psi_1, \dots, \psi_p \} \subseteq \Hsym(\widehat A)$ 
   where $\psi_1, \dots, \psi_p$ is a basis of eigenvectors or $\widehat A|_{H_{sa}(\widehat A)}$.
   Note that $0 \leq p \leq 2$.
   Clearly, each $\psi_j$ must satisfy \ref{item:b} and \ref{item:c}, and, if $\psi_j \neq 0$ it must also satisfy \ref{item:a}, otherwise $\psi\in H_{sa}(\Ssym) = \{0\}$.

   On the other hand, if $\phi$ satisfies \ref{item:a}, \ref{item:b} and \ref{item:c}, then it belongs to $H_{sa}(\widehat A)$, hence 
   $H_{sa}(\widehat A) \supseteq \gen\{\phi\}$ for any such $\phi$.
\end{proof}


\subsection{Application to quantum graphs: Star graph}
Our general set up is flexible enough to be applicable also to quantum graphs.
As a very simple example consider a finite star graph.

\begin{figure}[h]
	\begin{center}
		\begin{tikzpicture}[transform shape, baseline=(D)] 
			\coordinate (O) at (0,0);
			\coordinate (A) at (40:2);
			\coordinate (B) at (120:2);
			\coordinate (C) at (180:2);
			\coordinate (D) at (220:2);
			\coordinate (E) at (330:2);
			\coordinate (F) at (310:2);
			\foreach \x in {A,B,C,D, E} \draw (O) -- (\x);
			\foreach \x in {A,B,C,D,E,O}  \draw[fill] (\x) circle (0.05cm);
			\node at (A) [above=.1] {$f_1(x_1) = 0$};
			\node at (B) [above=.1] {$f_2(x_2) = 0$};
			\node at (C) [left=.1] {$f_3(x_3) = 0$};
			\node at (D) [below=.1] {$f_4(x_4) = 0$};
			\node at (E) [below] {$f_5(x_5) = 0$};
			\node at (O) [below=.1] {$0$};
		\end{tikzpicture}
	\end{center}
	\caption{A star graph $G$ with 5 edges, 5 outer vertices and one central vertex (identified with zero). We assume Dirichlet boundary conditions at the outer vertices and continuity at the central vertex.}
\label{Bischleben}
\end{figure}
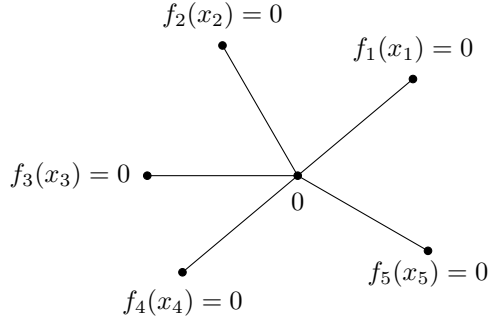

We consider a star graph $G$ consisting of $n$ edges $e_j = (0, x_j)$ which are joined at the vertex $x_0 = 0$.
The relevant Hilbert space is $H = \bigoplus_{j=1}^n L_2(0, x_j)$.
We write $f = (f_1, \dots, f_n)$ where $f_j$ is the restriction of $f$ to the edge $e_j$.
Let us define the subspaces
\begin{align}
    H^2(G) &= \bigoplus_{j=1}^n H^2(0, x_j), \\
    H_D(G) &= \{ f\in H^2(G) : f_j(x_j) = 0, f_j(0+) = f_1(0+),\, j=1,\dots, n \}.
\end{align}
So functions in $H_D(G)$ satisfy Dirichlet boundary conditions at the exterior vertices and are continuous at the central vertex.
The latter allows us to write $f(0)$.
We define the closed symmetric operator
\begin{align}
    S: \qquad
    \mD(S) = \Big\{f\in H_D(G) \ :\  \sum_{j=1}^n f_j'(0) = 0,\ f(0)= 0 \Big\},
    \qquad Sf = -f''.
\end{align}
It is easy to check that its adjoint is
\begin{align}
    S^*: \qquad
    \mD(S^*) = H_D(G),
    \qquad S^*f = -f''
\end{align}
and that $\eta_\pm(S)= 1$.
A boundary triple $(\Gamma_0, \Gamma_1, \C)$ for $S^*$ is given by 
\begin{align}
    \Gamma_0 f = f(0),
    \qquad
    \Gamma_1 f = -\sum_{j=1}^n f_j'(0).
\end{align}
Hence all one-dimensional extensions of $S$ between $S$ and $S^*$ are given by 
\begin{align}
    S_{b,c}: \qquad
    \mD(S_{b,c}) = \Big\{f\in H_D(G) \ :\  b f(0) =  -c\sum_{j=1}^n f_j'(0) \Big\},
    \qquad S_{b, c}f = -f''.
\end{align}
The extension is maximally dissipative if and only if $\im b\overline c \ge 0$.
It
is selfadjoint if and only if $b\overline c \in\R$.
In particular we have the Dirichlet extension $S_D$ for $c=0$.
Clearly, the spectrum of $S_D$ is discrete; it is the union of the spectra of the Dirichlet Laplacians $S_{j, D}$ on the edges $e_j$.
Therefore, the spectrum of $S_D$ is simple if and only if the lengths of the edges are incommensurable.

Let $V\ge 0$ be a bounded multiplication operator and let $\widehat A$ be a maximally dissipative extension of $A= S+ \I V$.
If the edges are incommensurable then $S$ is completely non-selfadjont and therefore $\dim( H_{sa}(\widehat A) ) \le 1$.

If  $V$ is a bounded multiplication such that on each edge $e_j$ there is a non-trivial interval $(a_j, b_j)$ such that $V(x)> 0$ for every $x\in (a_j,b_j)$, then, similar to Theorem~\ref{thm:5:15}, it follows that $\dim H_{sa}(\widehat A) \le 1$.

\appendix

\section{The imaginary part of a $2\times 2$ matrix}

For convenience of the reader, we collect some easy facts about positivity of $2\times 2$ matrices that are used frequently in Section~\ref{sec:dissipativeoperators}.

\begin{lemma}
   \begin{enumerate}[label={\upshape(\roman*)}]
      \item 
      Let $\alpha, \delta \in\R$, $\gamma\in \C$.
      The eigenvalues of the symmetric matrix
      $M = 
      \begin{pmatrix}
	 \alpha & \gamma \\ \overline{\gamma} & \delta
      \end{pmatrix}$
      are
      \begin{equation*}
	 \lambda_{\pm} = \frac{1}{2} (\alpha + \delta) \pm \frac{1}{2} \sqrt{ (\alpha-\delta)^2 - 4|\gamma|^2}
      \end{equation*}
      and
      \begin{equation*}
	 \lambda_{\pm} \ge 0
	 \qquad\iff\qquad
	 \alpha+\delta \ge 0
	 \quad\text{and}\quad
	 \alpha\delta \ge |\gamma|^2.
      \end{equation*}

      \item 
      The imaginary part of
      $A = \begin{pmatrix}
	 a & b \\ c & d
      \end{pmatrix}$
      is
      \begin{equation}
        \im A = \frac{1}{2\I } (A-A^*)
        =
        \begin{pmatrix}
	 \im   a & \frac{1}{2\I} (b-\overline c) \\ \frac{1}{2\I} (c-\overline b) & \im d
        \end{pmatrix}
        =
        \im
        \begin{pmatrix}
	 a &   \frac{b-\overline c}{2} \\  \frac{c-\overline b}{2} & d
        \end{pmatrix}.
      \end{equation}
      
      \item 
      $\im A \ge 0$
      \parbox[t]{.7\textwidth}{%
      \mbox{}$\quad\iff\quad
      \im(a+c) \ge 0
      \quad\text{and}\quad
      \im(a)\im(d) \ge \frac{1}{4} |b-c|^2$
      \\[1ex]
      \mbox{}$\quad\iff\quad
      \im a\ge 0,\ \im d\ge 0
      \quad\text{and}\quad
      \im(a)\im(d) \ge \frac{1}{4} |b-c|^2$.
      }
   \end{enumerate}
\end{lemma}

\addcontentsline{toc}{section}{Bibliography}
\bibliography{biblioFPW2025}
\bibliographystyle{alpha}

\end{document}